\documentclass{amsart}
\usepackage{amscd}
\usepackage{amsfonts}
\usepackage{amsmath}
\usepackage{amssymb}
\usepackage{latexsym}
\usepackage{amsthm,color}
\usepackage[all]{xy}
\usepackage{epsfig}
\usepackage{graphicx}
\usepackage{color}
\usepackage{tikz}
\usetikzlibrary{calc}
\newtheorem{proposition}{Proposition}

\newtheorem{theorem}{Theorem}

\theoremstyle{definition}
\newtheorem{definition}{Definition}
\theoremstyle{definition}
\newtheorem{example}{Example}
\theoremstyle{definition}
\newtheorem{remark}{Remark}

\newtheorem{problem}{Problem}

\newcommand{\bm}{\mathbf}

\begin{document}
\title[
Envelopes created by curve families in the Lorentz-Minkowski plane]
{Envelopes created by curve families in the Lorentz-Minkowski plane}
%%%%%%%%%%%%%%%%%%%%%%%%%%%%%%%%%%%%%%%%%%%%%%%%%%%%%%%%%%%%%%%%%%%
\author[M.~Takahashi]{Masatomo Takahashi%\thanks{Corresponding author.}
}
\address{
Muroran Institute of Technology,
Muroran 050-8585, Japan}
\email{masatomo@muroran-it.ac.jp}

\author[Y.~Wang]{Yongqiao Wang%\thanks{Corresponding author.}
}
\address{
School of Science, Dalian Maritime University, Dalian 116026, P.R. China
}
\email{wangyq@dlmu.edu.cn}
%%%%%%%%%%%%%%%%%%%%%%%%%%%%%%%%%%%%%%%%%%%%%%

%%%%%%%%%%%%%%%%%%%%%%%%%%%%%%%%%%%%%%%%%%%%%%%
%%%%%%%%%%%%%%%%%%%%%%%%%%%%%%%%%%%%%%%%%%%%%%%
\begin{abstract}
Classical definitions of envelopes are vague for singular curves in the Lorentz-Minkowski plane. To study envelopes of such curves, we introduce a family of smooth curves that may admit singular points. Notably, the curves considered here may contain both non-lightlike and lightlike points. We then propose a new definition of envelopes for such curves and investigate their properties using lightlike tangential data. Furthermore, we examine the contact between the curves and their envelopes at singular points.
\end{abstract}
 %, 53A40, 53A04}

%\thanks{}

\date{}

\maketitle
\section{Introduction}
Envelopes are classical objects in differential geometry, with numerous applications to differential equations, differential geometry, and physics (see, e.g., \cite{BG1,brucegiblin,EFK,izumiya1,nishimura,Takahashi1}). The envelope of a family of plane curves is a curve that is tangent to each member of the family. For regular curves, the notion of tangency is well defined. However, classical definitions of envelopes become inadequate for singular plane curves, since tangents are not well defined at singular points. To address this issue, the first author introduced a one-parameter family of Legendre curves in the unit tangent bundle over the Euclidean plane and proposed a definition of an envelope for such families in \cite{Takahashi}. Studying Legendre curves is a typical approach to investigating singular curves \cite{fukunagatakahashi}. It was shown that the envelope of a Legendre curve family is also a Legendre curve.

Lorentz-Minkowski spacetime provides the mathematical setting for the theory of relativity. In the Lorentz-Minkowski plane, non-lightlike curves that consist of only one type of point, namely spacelike or timelike points, can be studied using the Frenet frame similarly to the Euclidean case. However, this frame fails for mixed type curves, which contain points of different types. In \cite{izumiya}, the first author et al. introduced a lightcone frame for mixed type curves in the Lorentz-Minkowski plane, enabling the study of their evolutes. In \cite{CT,LP}, this approach was extended to mixed type curves in Minkowski 3-space and the corresponding fundamental theorem was established.
For mixed type surfaces, a Gauss-Bonnet type formula has been investigated in \cite{HST}.

In this paper, we aim to clarify the definition of envelopes for curves in the Lorentz-Minkowski plane that may admit singular points as well as non-lightlike and lightlike points. Lightlike points occur when a curve transitions between spacelike and timelike regions; moreover, closed regular curves in the Lorentz-Minkowski plane necessarily possess at least four such points. Consequently, a global moving frame cannot be defined in the same way as for Legendre curves. To address this, we introduce a lightcone frame and the lightlike tangential data for one-parameter families of curves, and propose a new definition of envelopes for such families. These one-parameter families generalize both families of mixed type curves and families of frontals in the Lorentz-Minkowski plane.

Section \ref{S2} briefly reviews the basic concepts of the Lorentz-Minkowski plane and the classical definition of envelopes given by implicit functions \cite{BG1,GAS}. Section \ref{section3} introduces one-parameter families of curves in the Lorentz-Minkowski plane, defines the associated lightlike tangential data for such families, and establishes an existence and uniqueness theorem. Section \ref{section4} proposes a new definition of envelopes for one-parameter families of curves. An equivalent condition for such envelopes is given in terms of the lightlike tangential data. Moreover, we clarify the relationship between the new envelopes and classical envelopes. Finally, Section \ref{Section5} investigates the contact between the envelope and the members of the curve family at the singular points of the envelope.

All manifolds and maps considered in this paper are differentiable of class $C^{\infty}$.

\section{Preliminaries}\label{S2}
Let $\mathbb{R}^2=\{(x,y)|x,y\in\mathbb{R}\}$ be the two-dimensional real vector space. The {\it Lorentz-Minkowski plane} $\mathbb{R}_1^2$ is $\mathbb{R}^2$ endowed with the pseudo-scalar product
$$\langle\cdot, \cdot\rangle=-dx^2+dy^2.$$
A non-zero vector $\bm{a}=(a_1,a_2)\in\mathbb{R}_1^2$ is said to be {\it timelike, lightlike} or {\it spacelike} if $\langle\bm{a},\bm{a}\rangle=-a_1^2+a_2^2$ is negative, zero, or positive, respectively.
%The zero vector is usually considered as a spacelike vector.
%Let $\lambda\in\mathbb{R}_+$ be a positive real number. Three types of pseudo-circles are defined as follows:
%\begin{align*}
%H^1(\bm{a},-\lambda)=&\{\bm{x}\in\mathbb{R}_1^2\mid\langle\bm{x}-\bm{a},\bm{x}-\bm{a}\rangle=-\lambda^2\},\\
%LC(\bm{a},0)=&\{\bm{x}\in\mathbb{R}_1^2\mid\langle\bm{x}-\bm{a},\bm{x}-\bm{a}\rangle=0\},\\
%S^1_1(\bm{a},\lambda)=&\{\bm{x}\in\mathbb{R}_1^2\mid\langle\bm{x}-\bm{a},\bm{x}-\bm{a}\rangle=\lambda^2\}.
%\end{align*}
%
Let $\widetilde{\bm{\gamma}}:I\rightarrow\mathbb{R}_1^2$ be a smooth curve, where $I$ is an interval of $\mathbb{R}$. A point $t_0\in I$ (or $\widetilde{\bm{\gamma}}(t_0)$) is called a {\it singular point} of $\widetilde{\bm{\gamma}}$ if $\dot{\widetilde{\bm{\gamma}}}(t_0)=(d\widetilde{\bm{\gamma}}/dt)(t_0)=\bm{0}$.  We say that $t_0\in I$ (or $\widetilde{\bm{\gamma}}(t_0)$) is {\it timelike, lightlike,} or {\it spacelike} if $\dot{\widetilde{\bm{\gamma}}}(t_0)$ is a timelike, lightlike, or spacelike vector, respectively. Moreover, the curve
$\widetilde{\bm{\gamma}}$ is said to be {\it timelike, lightlike} or {\it spacelike} if $\dot{\widetilde{\bm{\gamma}}}(t)$ is of the corresponding type for all $t\in I$. However, a smooth curve in the Lorentz-Minkowski plane may not belong to any of these three categories. We call $\widetilde{\bm{\gamma}}$ a {\it mixed type curve} if it contains points of different types among timelike, lightlike, or spacelike points.  More often than not, a smooth curve in $\mathbb{R}_1^2$ is a mixed type curve. For instance, in \cite{izumiya}, the astroid curve $\widetilde{\bm{\gamma}}:[0,2\pi)\rightarrow\mathbb{R}_1^2$ defined by
\begin{align*}
\widetilde{\bm{\gamma}}(t)=\left(2\cos^3t,2\sin^3t\right),
\end{align*}
was shown to be the evolute of the unit Euclidean circle in $\mathbb{R}_1^2$ (see Fig. \ref{figure_1}).
\begin{figure}[h]
\begin{center}
\includegraphics[width=5cm]
{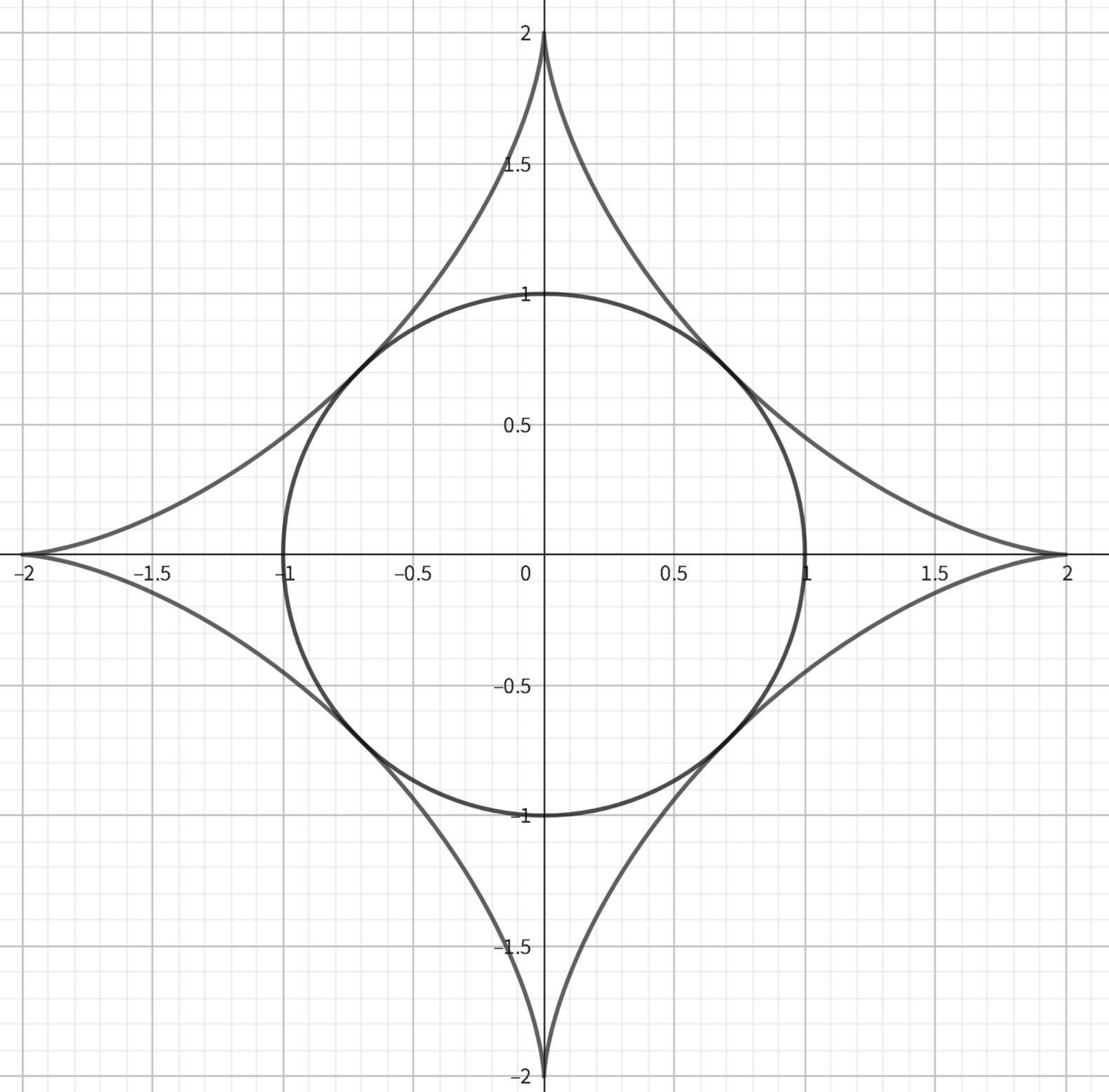}

\caption{The evolute of a circle is an astroid curve in $\mathbb{R}_1^2$.
}
\label{figure_1}
\end{center}
\end{figure}
It is worth noting that the evolute of the unit circle in $\mathbb{R}^2$ is a single point. A direct computation gives $$\langle\dot{\widetilde{\bm{\gamma}}}(t),\dot{\widetilde{\bm{\gamma}}}(t)\rangle=-\frac{3}{2}\sin4t.$$ Hence, $\widetilde{\bm{\gamma}}$ contains timelike, lightlike, and spacelike points. Moreover, $t=0,\pi/2,\pi,3\pi/2$ are singular points of $\widetilde{\bm{\gamma}}$.

We recall the classical definition of envelopes for one-parameter families of curves in $\mathbb{R}_1^2$ given by implicit functions. Let $F:V\times\Lambda\rightarrow\mathbb{R},$ $(x,y,\lambda)\mapsto F(x,y,\lambda)$ be a smooth function, where $V\subset\mathbb{R}^2_1$ is a domain and $\Lambda\subset\mathbb{R}$ is an interval. For each $\lambda\in\Lambda$, a family of curves in $\mathbb{R}_1^2$ is defined by
$$\Gamma_\lambda=\{(x,y)\in V\mid F(x,y,\lambda)=0\}.$$
The classical definition of the envelope is as follows.

\begin{definition}[$\mathcal{D}$ envelope \cite{brucegiblin}]\label{classical definition}
{\rm
The following set is called
the \textit{$\mathcal{D}$ envelope} of the family
$\Gamma_\lambda$ and is denoted by
\[\mathcal{D}=
\left\{(x, y)\in V\, \mid
\exists\lambda\in\Lambda\mbox{ such that } F(x, y, \lambda)=\frac{\partial F}{\partial \lambda}(x, y, \lambda)=0
\right\}.
\]
}
\end{definition}

We consider an example of families of astroid curves in the Lorentz-Minkowski plane.
\begin{example}\label{example astroid}
Let $F:\mathbb{R}^2_1\times(0,2\pi]\rightarrow\mathbb{R}$ be defined by $$F(x,y,\lambda)=(x\cos\lambda+y\sin\lambda)^{\frac{2}{3}}+(-x\sin\lambda+y\cos\lambda)^{\frac{2}{3}}-2^{\frac{2}{3}}.$$
A family of astroid curves in $\mathbb{R}_1^2$ is given by
$$\Gamma_\lambda=\{(x,y)\in \mathbb{R}^2_1|F(x,y,\lambda)=0\}.$$
We then have
\begin{align*}
\mathcal{D} = &
\left\{(x, y)\in \mathbb{R}^2_1\: \left|\: \exists \lambda\in (0,2\pi] \mbox{ s.t. }
F(x, y, \lambda)=\frac{\partial F}{\partial \lambda}(x, y, \lambda)=0\right.\right\} \\
= &
\left\{(x, y)\in \mathbb{R}^2_1\: \left|\: (x\cos\lambda+y\sin\lambda)(-x\sin\lambda+y\cos\lambda)\neq0,~
F(x, y, \lambda)=\frac{\partial F}{\partial \lambda}(x, y, \lambda)=0\right.\right\} \\ &\cup
\left\{(x, y)\in \mathbb{R}^2_1\: \left|\: (x\cos\lambda+y\sin\lambda)(-x\sin\lambda+y\cos\lambda)=0,~
F(x, y, \lambda)=\frac{\partial F}{\partial \lambda}(x, y, \lambda)=0\right.\right\} \\
 = &
\left\{(x, y)\in \mathbb{R}^2_1\: \left|\: x^2\cos^2\lambda+y^2\sin^2\lambda+2xy\cos\lambda\sin\lambda=
x^2\sin^2\lambda+y^2\cos^2\lambda-2xy\cos\lambda\sin\lambda=\frac{1}{2}\right.\right\} \\
  & \cup
\left\{(x, y)\in \mathbb{R}^2_1\: \left|\: x=2\cos\lambda, y=2\sin\lambda \mbox{ or } x=-2\sin\lambda, y=2\cos\lambda\right.\right\} \\
= & \left\{(x, y)\in \mathbb{R}^2_1\:
\left|\: x^2+y^2=1 \mbox{ or } x^2+y^2=4\right.\right\}.
\end{align*}
The $\mathcal{D}$ envelope of the family of astroid curves consists of two circles, as shown in Fig. \ref{figure_example2}.
\end{example}

\begin{figure}[h]
\begin{center}
\includegraphics[width=6cm]
{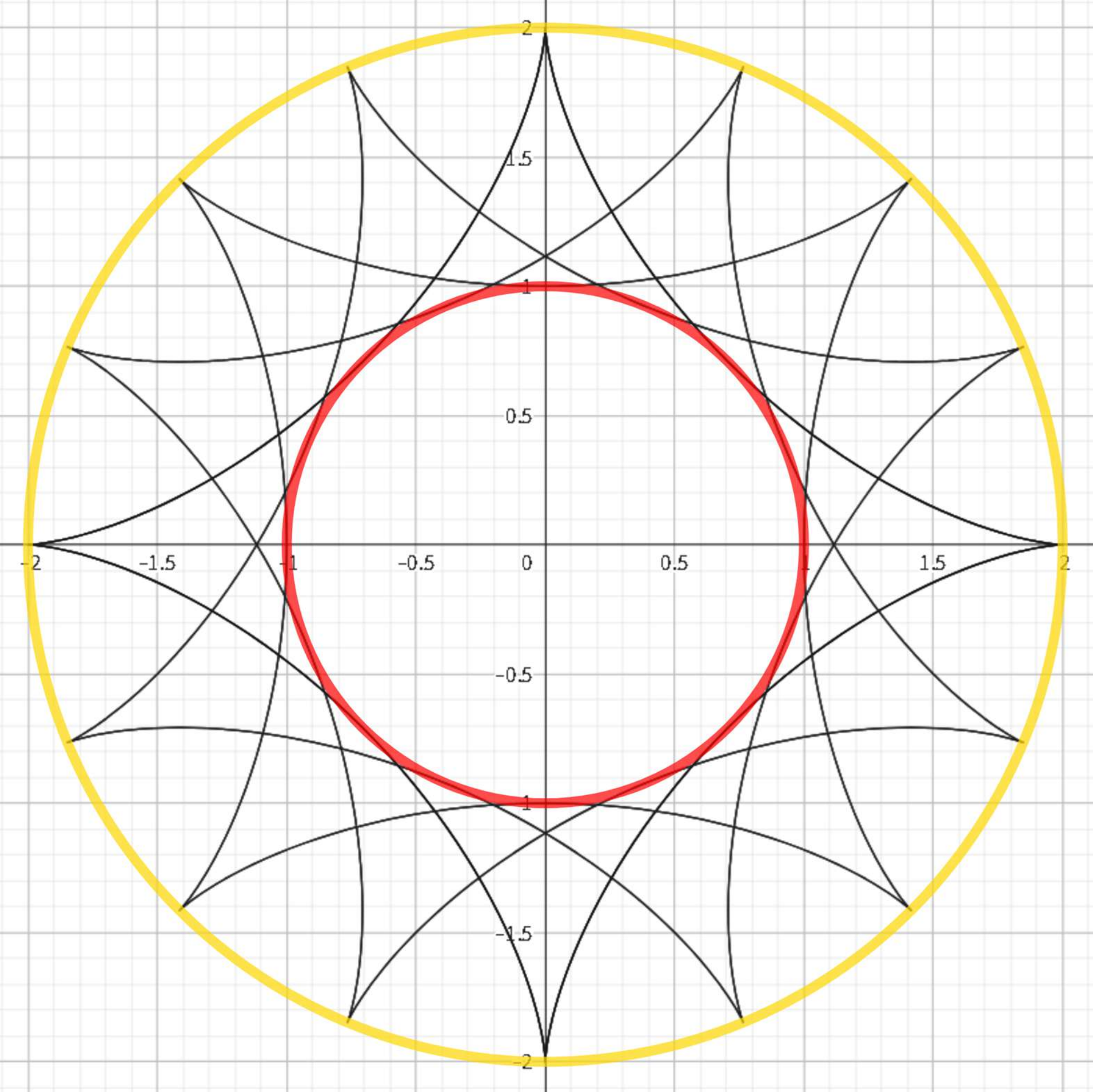}

\caption{One-parameter family of astroid curves
and the candidate of its envelope.
}
\label{figure_example2}
\end{center}
\end{figure}

\begin{problem}\label{p1}
In Example \ref{example astroid}, in the sense of limit tangent at the singularities of the astroid curves, the circle $\{(x,y)\in\mathbb{R}^2_1\mid x^2+y^2=4\}$  is not tangent to the astroid curves. Therefore, it is necessary to clarify the definition of envelopes for curves in the Lorentz-Minkowski plane
that may admit singular points as well as non-lightlike and lightlike points.
\end{problem}

\section{One-parameter families of curves in the Lorentz-Minkowski plane\label{section3}}
In this section, we consider one-parameter families of curves in the Lorentz-Minkowski plane. Let $\mathbb{L}^+=(1,1)$ and $\mathbb{L}^-=(1,-1)$. It is clear  that $\langle\mathbb{L}^+,\mathbb{L}^-\rangle=-2$, and $\mathbb{L}^+$, $\mathbb{L}^-$ are two linearly independent lightlike vectors.
Let $\widetilde{\bm{\gamma}}:I\rightarrow\mathbb{R}_1^2$ be a smooth curve. There exists a smooth mapping
$(\alpha,\beta):I\rightarrow\mathbb{R}^2$ such that
\[\dot{\widetilde{\bm{\gamma}}}(t)=\alpha(t)\mathbb{L}^++\beta(t)\mathbb{L}^-.\]
We call $\{\mathbb{L}^+,\mathbb{L}^-\}$ a {\it lightcone frame}, and the pair $(\alpha,\beta)$ the {\it lightlike tangential
 data} of $\widetilde{\bm{\gamma}}$. By definition, we have $$\langle\dot{\widetilde{\bm{\gamma}}}(t),\dot{\widetilde{\bm{\gamma}}}(t)\rangle=-4\alpha(t)\beta(t).$$ Hence, a regular point $\widetilde{\bm{\gamma}}(t_0)$ is timelike, lightlike, or spacelike if $\alpha(t_0)\beta(t_0)$ is positive, zero, or negative, respectively. A point $\widetilde{\bm{\gamma}}(t_0)$ is  singular  if $\alpha(t_0)=\beta(t_0)=0.$

Let $\bm{\gamma}: I\times\Lambda\to \mathbb{R}^2_1$ be a smooth mapping. If $\bm{\gamma}(\cdot,\lambda): I\to \mathbb{R}^2_1$ is a mixed type curve for each fixed $\lambda\in\Lambda$,
%i.e., $\bm{\gamma}(\cdot,\lambda)$ is an integrable curve in $\mathbb{R}^2_1$,
then $\bm{\gamma}$ is called a
{\it one-parameter family of mixed type curves}.
If $\bm{\gamma}(\cdot,\lambda): I\to \mathbb{R}^2_1$ is a singular curve for each fixed $\lambda\in\Lambda$,
then $\bm{\gamma}$ is called a
{\it one-parameter family of singular curves}.
There exists a smooth mapping $(\alpha,\beta,m,n):I\times\Lambda\rightarrow\mathbb{R}^4$ such that
\begin{align*}
\bm{\gamma}_t(t,\lambda)=\alpha(t,\lambda)\mathbb{L}^++\beta(t,\lambda)\mathbb{L}^-, \quad
\bm{\gamma}_\lambda(t,\lambda)&=m(t,\lambda)\mathbb{L}^++n(t,\lambda)\mathbb{L}^-,
\end{align*}
where $\bm{\gamma}_t(t,\lambda)=(\partial\bm{\gamma}/\partial t)(t,\lambda)$ and $\bm{\gamma}_\lambda(t,\lambda)=(\partial\bm{\gamma}/\partial\lambda)(t,\lambda)$.
From the integrability condition $\bm{\gamma}_{t\lambda}(t,\lambda)=\bm{\gamma}_{\lambda t}(t,\lambda)$, the functions $\alpha$, $\beta$, $m$, $n$ satisfy
\begin{align}\label{ingtegrability}
\alpha_\lambda(t,\lambda)=m_t(t,\lambda), \quad \beta_\lambda(t,\lambda)=n_t(t,\lambda)
\end{align}
for all $(t,\lambda)\in I\times\Lambda$.
We call the quadruple $(\alpha,\beta,m,n)$ with the integrability condition (\ref{ingtegrability})
the {\it lightlike tangential data} of the one-parameter family of curves $\bm{\gamma}$.

\begin{example}\label{example astroid1} (Example \ref{example astroid} revisited).
Let $\bm{\gamma}: [0,2\pi)\times[0,2\pi)\to \mathbb{R}^2_1$ be a one-parameter family of astroid curves defined by $$\bm{\gamma}(t,\lambda)=(2\cos\lambda\cos^3t-2\sin\lambda\sin^3t,2\sin\lambda\cos^3t+2\cos\lambda\sin^3t).$$
Since
\begin{align*}
\bm{\gamma}_t(t,\lambda)&=6\sin t\cos t\left(-\cos\lambda\cos t-\sin\lambda\sin t,-\sin\lambda\cos t+\cos\lambda\sin t\right),\\
\bm{\gamma}_\lambda(t,\lambda)&=2\left(-\sin\lambda\cos^3t-\cos\lambda\sin^3 t,\cos\lambda\cos^3 t-\sin\lambda\sin^3t\right),
\end{align*}
we obtain the expressions in terms of the lightcone frame $\{\mathbb{L}^+,\mathbb{L}^-\}$:
\begin{align*}
\bm{\gamma}_t(t,\lambda)=&\alpha(t,\lambda)\mathbb{L}^++\beta(t,\lambda)\mathbb{L}^-\\
=&3\sin t\cos t\left[-(\cos\lambda+\sin\lambda)\cos t+(\cos\lambda-\sin\lambda)\sin t\right]\mathbb{L}^+\\
&+3\sin t\cos t\left[(-\cos\lambda+\sin\lambda)\cos t-(\cos\lambda+\sin\lambda)\sin t\right]\mathbb{L}^-,\\
\bm{\gamma}_\lambda(t,\lambda)=&m(t,\lambda)\mathbb{L}^++n(t,\lambda)\mathbb{L}^-\\
=&\left[(\cos\lambda-\sin\lambda)\cos^3 t-(\cos\lambda+\sin\lambda)\sin^3 t\right]\mathbb{L}^+\\
&+\left[-(\cos\lambda+\sin\lambda)\cos^3 t-(\cos\lambda-\sin\lambda)\sin^3 t\right]\mathbb{L}^-.
\end{align*}
It follows that
\begin{align*}
\alpha(t,\lambda)&=3\sin t\cos t\left[-(\cos\lambda+\sin\lambda)\cos t+(\cos\lambda-\sin\lambda)\sin t\right],\\
\beta(t,\lambda)&=3\sin t\cos t\left[(-\cos\lambda+\sin\lambda)\cos t-(\cos\lambda+\sin\lambda)\sin t\right],\\
m(t,\lambda)&=(\cos\lambda-\sin\lambda)\cos^3 t-(\cos\lambda+\sin\lambda)\sin^3 t,\\
n(t,\lambda)&=-(\cos\lambda+\sin\lambda)\cos^3 t-(\cos\lambda-\sin\lambda)\sin^3 t,
\end{align*}
and
\begin{align*}
&\alpha_\lambda(t,\lambda)=m_t(t,\lambda)=3\sin t\cos t\left[(\sin\lambda-\cos\lambda)\cos t-(\sin\lambda+\cos\lambda)\sin t\right],\\
&\beta_\lambda(t,\lambda)=n_t(t,\lambda)=3\sin t\cos t\left[(\sin\lambda+\cos\lambda)\cos t+(\sin\lambda-\cos\lambda)\sin t\right].
\end{align*}
Thus the quadruple $(\alpha,\beta,m,n)$ is the lightlike tangential data of $\bm{\gamma}$. Moreover, we calculate that
\begin{align*}
\alpha(t,\lambda)\beta(t,\lambda)=\frac{9}{8}(1-\cos4t)\cos2(t-\lambda).
\end{align*}
For any fixed $\lambda\in[0,2\pi)$, we have $\alpha(t,\lambda)=\beta(t,\lambda)=0$ when $t=0,\pi/2,\pi,3\pi/2$; consequently,
$\bm{\gamma}$ is a one-parameter family of singular
curves. The sign of $\alpha(t,\lambda)\beta(t,\lambda)$ changes around  those $t\in[0,2\pi)$ where $\cos2(t-\lambda)=0;$ hence, $\bm{\gamma}$ is a one-parameter family of mixed type curves. Note that when $\lambda=\pi/4,3\pi/4,5\pi/4,7\pi/4$, although $\bm{\gamma}(\cdot,\lambda)$ is a closed curve, it contains no lightlike points.
\end{example}

\begin{example}\label{example closed curves}
Let $\bm{\gamma}: [0,2\pi)\times[0,2\pi)\to \mathbb{R}^2_1$ be a family of closed curves defined by
\begin{align*}
\bm{\gamma}(t,\lambda)=\left(\frac{1}{3}\cos\lambda\cos^3t+\frac{1}{3}\sin\lambda\sin^3t-\sin\lambda\sin t,~ \frac{1}{3}\sin\lambda\cos^3t-\frac{1}{3}\cos\lambda\sin^3t+\cos\lambda\sin t\right),
\end{align*}
(see Fig. \ref{figure_example3}).
The partial derivatives are
\begin{align*}
\bm{\gamma}_t(t,\lambda)=(&-\cos\lambda\cos^2t\sin t+\sin\lambda\sin^2t\cos t-\sin\lambda\cos t,~ -\sin\lambda\cos^2t\sin t-\cos\lambda\sin^2t\cos t
\\&+\cos\lambda\cos t ),\\
\bm{\gamma}_\lambda(t,\lambda)=\bigg(&-\frac{1}{3}\sin\lambda\cos^3t+\frac{1}{3}\cos\lambda\sin^3t-\cos\lambda\sin t,~ \frac{1}{3}\cos\lambda\cos^3t+\frac{1}{3}\sin\lambda\sin^3t-\sin\lambda\sin t\bigg).
\end{align*}
We then obtain
\begin{align*}
\bm{\gamma}_t(t,\lambda)=&\alpha(t,\lambda)\mathbb{L}^++\beta(t,\lambda)\mathbb{L}^-\\
=&-\frac{1}{2}\cos^2t\left[\cos\lambda(\sin t-\cos t)+\sin\lambda(\sin t+\cos t)\right]\mathbb{L}^+\\
&+\frac{1}{2}\cos^2t\left[\sin\lambda(\sin t-\cos t)-\cos\lambda(\sin t+\cos t)\right]\mathbb{L}^-,\\
\bm{\gamma}_\lambda(t,\lambda)=&m(t,\lambda)\mathbb{L}^++n(t,\lambda)\mathbb{L}^-\\
=&\left[\frac{1}{6}(\cos\lambda-\sin\lambda)\cos^3t-\frac{1}{2}(\cos\lambda+\sin\lambda)\sin t+\frac{1}{6}(\cos\lambda+\sin\lambda)\sin^3t\right]\mathbb{L}^+\\
&+\left[-\frac{1}{6}(\cos\lambda+\sin\lambda)\cos^3t-\frac{1}{2}(\cos\lambda-\sin\lambda)\sin t+\frac{1}{6}(\cos\lambda-\sin\lambda)\sin^3t\right]\mathbb{L}^-.
\end{align*}
It follows that
\begin{align*}
\alpha(t,\lambda)&=-\frac{1}{2}\cos^2t\left[\cos\lambda(\sin t-\cos t)+\sin\lambda(\sin t+\cos t)\right],\\
\beta(t,\lambda)&=\frac{1}{2}\cos^2t\left[\sin\lambda(\sin t-\cos t)-\cos\lambda(\sin t+\cos t)\right],\\
m(t,\lambda)&=\frac{1}{6}(\cos\lambda-\sin\lambda)\cos^3t-\frac{1}{2}(\cos\lambda+\sin\lambda)\sin t+\frac{1}{6}(\cos\lambda+\sin\lambda)\sin^3t,\\
n(t,\lambda)&=-\frac{1}{6}(\cos\lambda+\sin\lambda)\cos^3t-\frac{1}{2}(\cos\lambda-\sin\lambda)\sin t+\frac{1}{6}(\cos\lambda-\sin\lambda)\sin^3t,
\end{align*}
and
\begin{align*}
&\alpha_\lambda(t,\lambda)=m_t(t,\lambda)=\frac{1}{2}\cos^2t\left[-\cos\lambda(\sin t+\cos t)+\sin\lambda(\sin t-\cos t)\right],\\
&\beta_\lambda(t,\lambda)=n_t(t,\lambda)=\frac{1}{2}\cos^2t\left[\cos\lambda(\sin t-\cos t)+\sin\lambda(\sin t+\cos t)\right].
\end{align*}
Thus the quadruple $(\alpha,\beta,m,n)$ is the lightlike tangential data of $\bm{\gamma}$. Moreover, we calculate that
\begin{align*}
\alpha(t,\lambda)\beta(t,\lambda)=-\frac{1}{4}\cos^4t\cos2(t+\lambda).
\end{align*}
For any fixed $\lambda\in[0,2\pi)$, we have $\alpha(t,\lambda)=\beta(t,\lambda)=0$ when $t=\pi/2,3\pi/2$; consequently,
$\bm{\gamma}$ is a one-parameter family of singular
curves. The sign of $\alpha(t,\lambda)\beta(t,\lambda)$ changes around  those $t\in[0,2\pi)$ where $\cos2(t+\lambda)=0;$ hence, $\bm{\gamma}$ is a one-parameter family of mixed type curves. Note that for any fixed $\lambda\in[0,2\pi)$, $\bm{\gamma}(\cdot,\lambda)$ is a closed curve contains lightlike points.
\begin{figure}[h]
\begin{center}
\includegraphics[width=6cm]
{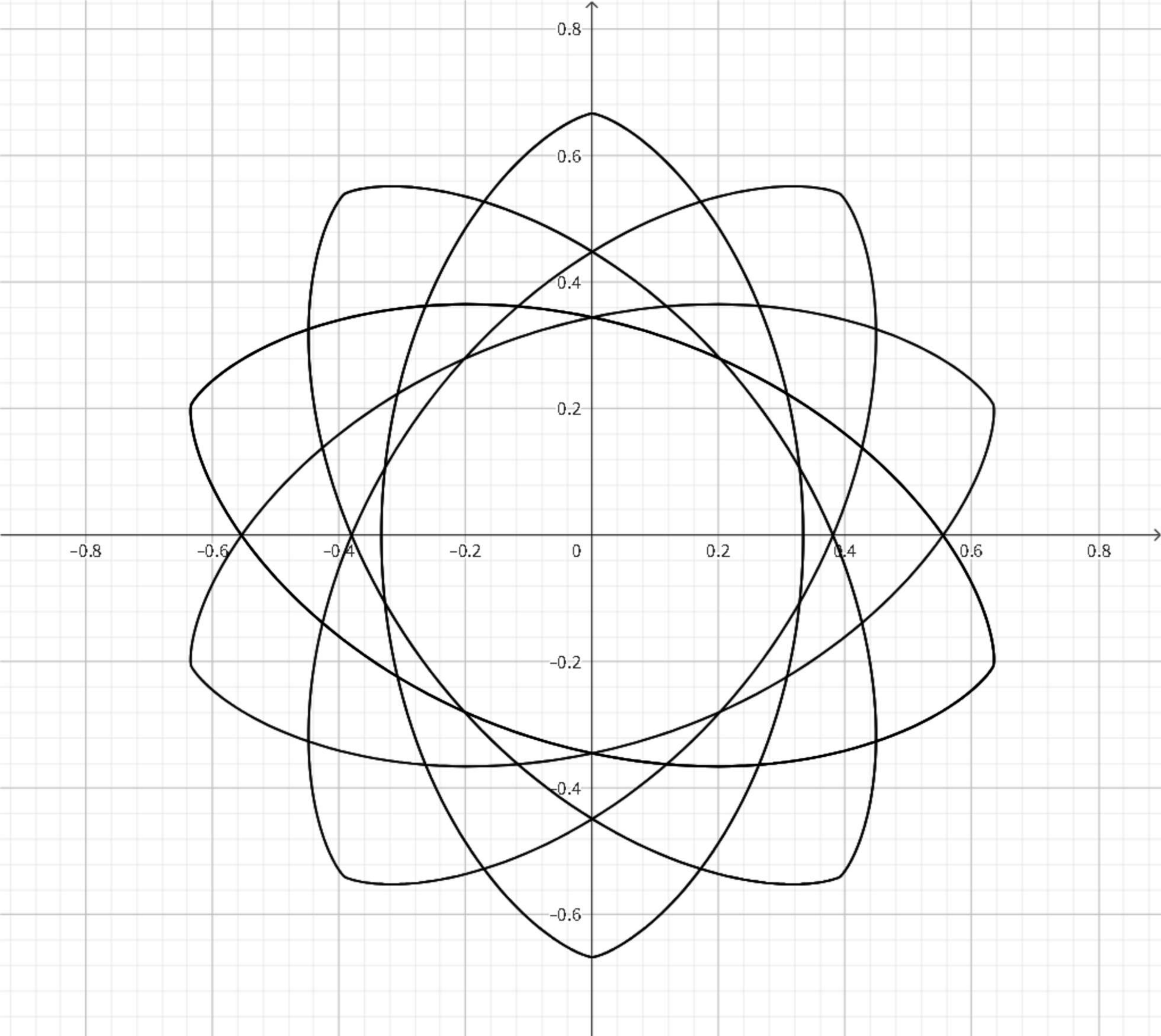}

\caption{One-parameter family of closed curves in $\mathbb{R}_1^2$.
}
\label{figure_example3}
\end{center}
\end{figure}

\end{example}

We now establish the existence and uniqueness theorems for one-parameter families of curves with given lightlike tangential data $(\alpha,\beta,m,n)$.

\begin{theorem}\label{theorem11}
(Existence Theorem) Let $(\alpha,\beta,m,n):I\times\Lambda\rightarrow\mathbb{R}^4$ be a smooth mapping satisfying the integrability condition (\ref{ingtegrability}). Then there exists a one-parameter family of curves $\bm{\gamma}: I\times\Lambda\to \mathbb{R}^2_1$ whose lightlike tangential data is $(\alpha,\beta,m,n)$.
\end{theorem}
\begin{proof}
Let $(t_0,\lambda_0)\in I\times\Lambda$ be fixed. Define $\bm{\gamma}: I\times\Lambda\to \mathbb{R}^2_1$  by
\begin{align*}
\bm{\gamma}(t,\lambda)=\bigg(&\int_{t_0}^{t}(\alpha(t,\lambda)+\beta(t,\lambda))dt+\int_{\lambda_0}^{\lambda}(m(t_0,\lambda)+n(t_0,\lambda))d\lambda,\\
&\int_{t_0}^{t}(\alpha(t,\lambda)-\beta(t,\lambda))dt+\int_{\lambda_0}^{\lambda}(m(t_0,\lambda)-n(t_0,\lambda))d\lambda\bigg).
\end{align*}
A direct calculation yields
\begin{align*}
\bm{\gamma}_t(t,\lambda)=\left(\alpha(t,\lambda)+\beta(t,\lambda),\alpha(t,\lambda)-\beta(t,\lambda)\right)
=\alpha(t,\lambda)\mathbb{L}^++\beta(t,\lambda)\mathbb{L}^-.
\end{align*}
Moreover, by the integrability condition (\ref{ingtegrability}), we obtain
\begin{align*}
\bm{\gamma}_\lambda(t,\lambda)=&\bigg(\int_{t_0}^{t}(\alpha_\lambda(t,\lambda)+\beta_\lambda(t,\lambda))dt+m(t_0,\lambda)+n(t_0,\lambda),\\
&\int_{t_0}^{t}(\alpha_\lambda(t,\lambda)-\beta_\lambda(t,\lambda))dt+m(t_0,\lambda)-n(t_0,\lambda)\bigg)\\
=&(m(t,\lambda)+n(t,\lambda),m(t,\lambda)-n(t,\lambda))\\
=&m(t,\lambda)\mathbb{L}^++n(t,\lambda)\mathbb{L}^-.
\end{align*}
Hence, $\bm{\gamma}$ is a one-parameter family of curves with lightlike tangential data $(\alpha,\beta,m,n)$.
\end{proof}

\begin{proposition}\label{prop1}
Let $\bm{\gamma},\overline{\bm{\gamma}}: I\times\Lambda\to \mathbb{R}^2_1$ be one-parameter families of curves with the same lightlike tangential data $(\alpha,\beta,m,n)$. Then there exists a constant $\bm{c}\in\mathbb{R}^2_1$ such that
$\overline{\bm{\gamma}}(t,\lambda)=\bm{\gamma}(t,\lambda)+\bm{c}$.
\end{proposition}
\begin{proof}
A direct calculation shows that
\begin{align*}\frac{\partial}{\partial t}(\overline{\bm{\gamma}}(t,\lambda)-\bm{\gamma}(t,\lambda))=0, \quad \frac{\partial}{\partial \lambda}(\overline{\bm{\gamma}}(t,\lambda)-\bm{\gamma}(t,\lambda))=0.\end{align*}
Hence, $\overline{\bm{\gamma}}(t,\lambda)-\bm{\gamma}(t,\lambda)$ is constant, i.e., there exists $\bm{c}\in\mathbb{R}^2_1$ such that
$\overline{\bm{\gamma}}(t,\lambda)=\bm{\gamma}(t,\lambda)+\bm{c}$.
\end{proof}
\begin{remark}
The condition in Proposition \ref{prop1} is rather strong. We adopt a milder alternative to ensure uniqueness up to a Lorentz motion.
\end{remark}
\begin{definition}\label{Lorentz motion}
Let $\bm{\gamma},\overline{\bm{\gamma}}: I\times\Lambda\to \mathbb{R}^2_1$ be one-parameter families of curves. If there exist a matrix
 $A$ and a constant $\bm{c}\in\mathbb{R}^2_1$ such that
$
\overline{\bm{\gamma}}(t,\lambda)=A(\bm{\gamma}(t,\lambda))+\bm{c}$ for all $(t,\lambda)\in I\times\Lambda,$
where
\begin{align*}
A=\left(\begin{array}{cc}
          \cosh a & -\sinh a \\
          -\sinh a & \cosh a
        \end{array}
\right)
~~\mbox{ or }~~
A=-\left(\begin{array}{cc}
          \cosh a & -\sinh a \\
          -\sinh a & \cosh a
        \end{array}
\right)
\end{align*}
for some $a\in\mathbb{R}$, then we say that $\bm{\gamma}$ and $\overline{\bm{\gamma}}$ are {\it congruent through a Lorentz motion}.
\end{definition}

\begin{proposition}\label{prop2}
Let $\bm{\gamma},\overline{\bm{\gamma}}: I\times\Lambda\to \mathbb{R}^2_1$ be one-parameter families of curves with lightlike tangential data $(\alpha,\beta,m,n)$ and $(\overline{\alpha},\overline{\beta},\overline{m},\overline{n})$, respectively. Suppose that $\bm{\gamma}$ and $\overline{\bm{\gamma}}$ are congruent through a Lorentz motion, i.e., there exist a constant $\bm{c}\in\mathbb{R}^2_1$ and a matrix
\begin{align*}
A=\left(\begin{array}{cc}
          \cosh a & -\sinh a \\
          -\sinh a & \cosh a
        \end{array}
\right)
~~\left( \mathrm{or} ,~~
A=-\left(\begin{array}{cc}
          \cosh a & -\sinh a \\
          -\sinh a & \cosh a
        \end{array}
\right)\right)
\end{align*}
for some $a\in\mathbb{R}$, such that $\overline{\bm{\gamma}}(t,\lambda)=\bm{\gamma}(t,\lambda)+\bm{c}$. Then
\begin{align*}
\overline{\alpha}(t,\lambda)&=e^{-a}\alpha(t,\lambda),~~\overline{\beta}(t,\lambda)=e^{a}\beta(t,\lambda),
~~\overline{m}(t,\lambda)=e^{-a}m(t,\lambda),~~\overline{n}(t,\lambda)=e^{a}n(t,\lambda)\\
( \mathrm{or} ,~ \overline{\alpha}(t,\lambda)&=-e^{-a}\alpha(t,\lambda),~~\overline{\beta}(t,\lambda)=-e^{a}\beta(t,\lambda),
~~\overline{m}(t,\lambda)=-e^{-a}m(t,\lambda),~~\overline{n}(t,\lambda)=-e^{a}n(t,\lambda)).
\end{align*}
\end{proposition}
\begin{proof}
Assume that $\overline{\bm{\gamma}}(t,\lambda)=\bm{\gamma}(t,\lambda)+\bm{c}$. Differentiating with respect to $t$ gives
\begin{align*}
\overline{\bm{\gamma}}_t(t,\lambda)=A(\bm{\gamma}_t(t,\lambda))
=\alpha(t,\lambda)A(\mathbb{L}^+)+\beta(t,\lambda)A(\mathbb{L}^-).
\end{align*}
Since
\begin{align*}
A(\mathbb{L}^+)=\left(\begin{array}{cc}
          \cosh a & -\sinh a \\
          -\sinh a & \cosh a
        \end{array}
\right)
\left(\begin{array}{c}
        1 \\
        1
      \end{array}
\right)=(\cosh a-\sinh a)
\left(\begin{array}{c}
        1 \\
        1
      \end{array}
\right)=e^{-a}\mathbb{L}^+,
\end{align*}
\begin{align*}
A(\mathbb{L}^-)=\left(\begin{array}{cc}
          \cosh a & -\sinh a \\
          -\sinh a & \cosh a
        \end{array}
\right)
\left(\begin{array}{c}
        1 \\
        -1
      \end{array}
\right)=(\cosh a+\sinh a)
\left(\begin{array}{c}
        1 \\
        -1
      \end{array}
\right)=e^{a}\mathbb{L}^-,
\end{align*}
we obtain
\begin{align*}
\overline{\bm{\gamma}}_t(t,\lambda)%=e^{-\theta}(\alpha(t,\lambda)\mathbb{L}^++\beta(t,\lambda)\mathbb{L}^-)
=e^{-a}\alpha(t,\lambda)\mathbb{L}^++e^{a}\beta(t,\lambda)\mathbb{L}^-.
\end{align*}
Comparing with the decomposition $\overline{\bm{\gamma}}_t(t,\lambda)=\overline{\alpha}(t,\lambda)\mathbb{L}^++\overline{\beta}(t,\lambda)\mathbb{L}^-$,
we deduce
\begin{align*}
\overline{\alpha}(t,\lambda)=e^{-a}\alpha(t,\lambda),\quad \overline{\beta}(t,\lambda)=e^{a}\beta(t,\lambda).
\end{align*}
Similarly, differentiating with respect to $\lambda$ yields
\begin{align*}
\overline{\bm{\gamma}}_\lambda(t,\lambda)=A(\bm{\gamma}_\lambda(t,\lambda))
=m(t,\lambda)A(\mathbb{L}^+)+n(t,\lambda)A(\mathbb{L}^-)
=e^{-a}m(t,\lambda)\mathbb{L}^++e^{a}n(t,\lambda)\mathbb{L}^-,
\end{align*}
and therefore
\begin{align*}
\overline{m}(t,\lambda)=e^{-a}m(t,\lambda),\quad \overline{n}(t,\lambda)=e^{a}n(t,\lambda).
\end{align*}
The other case can be proved by the same way.
\end{proof}

\begin{theorem}\label{theorem12}
(Uniqueness Theorem) Let $\bm{\gamma},\overline{\bm{\gamma}}: I\times\Lambda\to \mathbb{R}^2_1$ be one-parameter families of curves with lightlike tangential data $(\alpha,\beta,m,n)$ and $(\overline{\alpha},\overline{\beta},\overline{m},\overline{n})$, respectively, both satisfying the integrability condition (\ref{ingtegrability}). Suppose that the set $$\Sigma=\{(t,\lambda)\in I\times\Lambda|\alpha(t,\lambda)\beta(t,\lambda)\neq0~ \mathrm{and}~m(t,\lambda)n(t,\lambda)\neq0\}$$
is dense in $I\times\Lambda$. If the following equalities hold for all $(t,\lambda)\in I\times\Lambda$:
 \begin{align*}
&\alpha_t(t,\lambda)\beta(t,\lambda)-\alpha(t,\lambda)\beta_t(t,\lambda)=
\overline{\alpha}_t(t,\lambda)\overline{\beta}(t,\lambda)-\overline{\alpha}(t,\lambda)\overline{\beta}_t(t,\lambda),\\
&m_t(t,\lambda)n(t,\lambda)-m(t,\lambda)n_t(t,\lambda)=
\overline{m}_t(t,\lambda)\overline{n}(t,\lambda)-\overline{m}(t,\lambda)\overline{n}_t(t,\lambda),\\
&\alpha_\lambda(t,\lambda)\beta(t,\lambda)-\alpha(t,\lambda)\beta_\lambda(t,\lambda)=
\overline{\alpha}_\lambda(t,\lambda)\overline{\beta}(t,\lambda)-\overline{\alpha}(t,\lambda)\overline{\beta}_\lambda(t,\lambda),\\
&m_\lambda(t,\lambda)n(t,\lambda)-m(t,\lambda)n_\lambda(t,\lambda)=
\overline{m}_\lambda(t,\lambda)\overline{n}(t,\lambda)-\overline{m}(t,\lambda)\overline{n}_\lambda(t,\lambda)
 \end{align*}
 and additionally
\begin{align*}
\alpha(t,\lambda)\beta(t,\lambda)=\overline{\alpha}(t,\lambda)\overline{\beta}(t,\lambda),~~
m(t,\lambda)n(t,\lambda)=\overline{m}(t,\lambda)\overline{n}(t,\lambda),~~
\alpha(t,\lambda)\overline{m}(t,\lambda)=m(t,\lambda)\overline{\alpha}(t,\lambda),
\end{align*}
then $\bm{\gamma}$ and $\overline{\bm{\gamma}}$ are congruent through a Lorentz motion.
\end{theorem}
\begin{proof}
Fix a point $(t_0,\lambda_0)\in \Sigma$. From the assumptions, we have
\begin{align*}
\alpha(t_0,\lambda_0)\beta(t_0,\lambda_0)=\overline{\alpha}(t_0,\lambda_0)\overline{\beta}(t_0,\lambda_0)\neq0,\quad
m(t_0,\lambda_0)n(t_0,\lambda_0)=\overline{m}(t_0,\lambda_0)\overline{n}(t_0,\lambda_0)\neq0.
 \end{align*}
Hence $\alpha(t_0,\lambda_0)\beta(t_0,\lambda_0)=\overline{\alpha}(t_0,\lambda_0)\overline{\beta}(t_0,\lambda_0)>0~\mathrm{or}<0$,
and $m(t_0,\lambda_0)n(t_0,\lambda_0)=\overline{m}(t_0,\lambda_0)\overline{n}(t_0,\lambda_0)>0~\mathrm{or}<0.$
Moreover, since $\alpha(t_0,\lambda_0)\overline{m}(t_0,\lambda_0)=m(t_0,\lambda_0)\overline{\alpha}(t_0,\lambda_0)$, there exists a Lorenz motion, i.e., a matrix
\begin{align*}
A=\left(\begin{array}{cc}
          \cosh a & -\sinh a \\
          -\sinh a & \cosh a
        \end{array}
\right)
\end{align*}
and a constant $\bm{c}\in\mathbb{R}^2_1$ such that
\begin{align*}
\overline{\bm{\gamma}}(t_0,\lambda_0)=\pm A(\bm{\gamma}(t_0,\lambda_0))+\bm{c},\quad
\overline{\bm{\gamma}}_t(t_0,\lambda_0)=\pm A(\bm{\gamma}_t(t_0,\lambda_0)),\quad
\overline{\bm{\gamma}}_\lambda(t_0,\lambda_0)=\pm A(\bm{\gamma}_\lambda(t_0,\lambda_0)).
\end{align*}
Now differentiate $\alpha(t,\lambda)\beta(t,\lambda)=\overline{\alpha}(t,\lambda)\overline{\beta}(t,\lambda)$ with respect to $t$:
\begin{align*}
\alpha_t(t,\lambda)\beta(t,\lambda)+\alpha(t,\lambda)\beta_t(t,\lambda)
=\overline{\alpha}_t(t,\lambda)\overline{\beta}(t,\lambda)+\overline{\alpha}(t,\lambda)\overline{\beta}_t(t,\lambda).
\end{align*}
Using the condition
\begin{align*}
\alpha_t(t,\lambda)\beta(t,\lambda)-\alpha(t,\lambda)\beta_t(t,\lambda)
=\overline{\alpha}_t(t,\lambda)\overline{\beta}(t,\lambda)-\overline{\alpha}(t,\lambda)\overline{\beta}_t(t,\lambda),
\end{align*}
we obtain $\alpha_t(t,\lambda)\beta(t,\lambda)=\overline{\alpha}_t(t,\lambda)\overline{\beta}(t,\lambda)$ and $\alpha(t,\lambda)\beta_t(t,\lambda)=\overline{\alpha}(t,\lambda)\overline{\beta}_t(t,\lambda)$.
Thus, for $(t,\lambda)\in \Sigma$, we have
\begin{align*}
\left(\begin{array}{cc}
        \alpha(t,\lambda) & \overline{\alpha}(t,\lambda) \\
        \alpha_t(t,\lambda) & \overline{\alpha}_t(t,\lambda)
      \end{array}
\right)
\left(\begin{array}{c}
        \beta(t,\lambda) \\
        -\overline{\beta}(t,\lambda)
      \end{array}
\right)=
\left(\begin{array}{c}
        0 \\
        0
      \end{array}
\right).
\end{align*}
Since $\alpha(t,\lambda)\neq0$ and $\beta(t,\lambda)\neq0$ on $\Sigma$. This implies that $$\alpha(t,\lambda)\overline{\alpha}_t(t,\lambda)-\alpha_t(t,\lambda)\overline{\alpha}(t,\lambda)=0$$ on $\Sigma,$
which is equivalent to
$$\frac{\partial}{\partial t}\bigg(\frac{\overline{\alpha}(t,\lambda)}{\alpha(t,\lambda)}\bigg)=0$$
on $\Sigma.$
By taking the partial derivative of $\alpha(t,\lambda)\beta(t,\lambda)=\overline{\alpha}(t,\lambda)\overline{\beta}(t,\lambda)$ with respect to $\lambda$ and applying the condition
\begin{align*}
\alpha_\lambda(t,\lambda)\beta(t,\lambda)-\alpha(t,\lambda)\beta_\lambda(t,\lambda)=
\overline{\alpha}_\lambda(t,\lambda)\overline{\beta}(t,\lambda)-\overline{\alpha}(t,\lambda)\overline{\beta}_\lambda(t,\lambda),
\end{align*}
we also have
$$\frac{\partial}{\partial \lambda}\bigg(\frac{\overline{\alpha}(t,\lambda)}{\alpha(t,\lambda)}\bigg)=0$$
on $\Sigma.$
Hence, there exists a constant $b\in\mathbb{R}$ such that $\overline{\alpha}(t,\lambda)=b\alpha(t,\lambda)$ on $\Sigma.$ Evaluating at  $(t_0,\lambda_0)$ and using the fact that $\overline{\alpha}(t_0,\lambda_0)=\pm e^{-a}\alpha(t_0,\lambda_0)$ (from the Lorentz motion at $(t_0,\lambda_0)$), we obtain $b=\pm e^{-a}$.
Consequently, we have that $\overline{\beta}(t,\lambda)=\pm e^{a}\beta(t,\lambda)$ on $\Sigma$ by the relation $\alpha(t,\lambda)\beta(t,\lambda)=\overline{\alpha}(t,\lambda)\overline{\beta}(t,\lambda)$. Since $\Sigma$ is dense in $I\times\Lambda$, the equalities extend by continuity to the whole domain $I\times\Lambda$:
\begin{align*}
\overline{\alpha}(t,\lambda)=\pm e^{-a}\alpha(t,\lambda),\quad
\overline{\beta}(t,\lambda)=\pm e^{a}\beta(t,\lambda).
\end{align*}
An analogous argument using the conditions
\begin{align*}
&m_t(t,\lambda)n(t,\lambda)-m(t,\lambda)n_t(t,\lambda)=
\overline{m}_t(t,\lambda)\overline{n}(t,\lambda)-\overline{m}(t,\lambda)\overline{n}_t(t,\lambda),\\
&m_\lambda(t,\lambda)n(t,\lambda)-m(t,\lambda)n_\lambda(t,\lambda)=
\overline{m}_\lambda(t,\lambda)\overline{n}(t,\lambda)-\overline{m}(t,\lambda)\overline{n}_\lambda(t,\lambda),
\end{align*}
and
\begin{align*}
m(t,\lambda)n(t,\lambda)=\overline{m}(t,\lambda)\overline{n}(t,\lambda),
\end{align*}
we can obtain
\begin{align*}
\overline{m}(t,\lambda)=\pm e^{-a}m(t,\lambda),\quad
\overline{n}(t,\lambda)=\pm e^{a}n(t,\lambda)
\end{align*}
for $(t,\lambda)\in I\times\Lambda$. Therefore,
\begin{align*}
\frac{\partial}{\partial t}\left(\overline{\bm{\gamma}}(t,\lambda)\mp A(\bm{\gamma}(t,\lambda))\right)=0,\quad
\frac{\partial}{\partial \lambda}\left(\overline{\bm{\gamma}}(t,\lambda)\mp A(\bm{\gamma}(t,\lambda))\right)=0,
\end{align*}
so $\overline{\bm{\gamma}}(t,\lambda)\mp A(\bm{\gamma}(t,\lambda))$ is constant.
Using the initial condition $\overline{\bm{\gamma}}(t_0,\lambda_0)=\pm A(\bm{\gamma}(t_0,\lambda_0))+\bm{c}$, we conclude
$$\overline{\bm{\gamma}}(t,\lambda)=\pm A(\bm{\gamma}(t,\lambda))+\bm{c}$$
for all $(t,\lambda)\in I\times\Lambda$. Thus, $\bm{\gamma}$ and $\overline{\bm{\gamma}}$ are congruent through a Lorentz motion.
\end{proof}

\section{Envelopes of curve families in the Lorentz-Minkowski plane\label{section4}}
Let $\bm{\gamma}: I\times\Lambda\to \mathbb{R}^2_1$ be a one-parameter family of curves with lightlike tangential data $(\alpha,\beta,m,n)$. In order to solve Problem \ref{p1}, throughout the rest of this paper, we assume that there exist smooth functions $\theta,\ell:I\times\Lambda\to \mathbb{R}$ such that
\begin{align*}
(\alpha(t,\lambda),\beta(t,\lambda))
=\ell(t,\lambda)(\cos\theta(t,\lambda),\sin\theta(t,\lambda))
\end{align*}
for all $(t,\lambda)\in I\times\Lambda$, and the set
$$\overline{\Sigma}=\{(t,\lambda)\in I\times\Lambda\mid\ell(t,\lambda)\neq0\}$$
is dense in $I\times\Lambda$.
%By the continuity of the function, the following lemma can be deduced.
%\begin{lemma}\label{l4.1}
%Assume that there exist smooth functions $\theta,\ell:I\times\Lambda\to \mathbb{R}$ satisfying
%\begin{align*}
%(\alpha(t,\lambda),\beta(t,\lambda))
%=\ell(t,\lambda)(\cos\theta(t,\lambda),\sin\theta(t,\lambda))
%\end{align*}
%on the set $\overline{\Sigma}=\{(t,\lambda)\in I\times\Lambda\mid(\alpha(t,\lambda),\beta(t,\lambda))\neq(0,0)\}$. Then the products $\ell\cos\theta$ and $\ell\sin\theta$ are unique if and only if $\overline{\Sigma}$ is dense in $I\times\Lambda$.
%\end{lemma}
%\begin{remark}
%For each fixed $\lambda\in\Lambda$, if the singular points of $\bm{\gamma}(\cdot,\lambda): I\to \mathbb{R}^2_1$ are isolated, then the condition that $\overline{\Sigma}$ is dense in $I\times\Lambda$ is satisfied.
%\end{remark}
%In order to solve Problem \ref{p1}, throughout the rest of this paper, we assume that such functions $\theta,\ell$ exist and $\overline{\Sigma}=\{(t,\lambda)\in I\times\Lambda\mid(\alpha(t,\lambda),\beta(t,\lambda))\neq(0,0)\}$ is dense in $I\times\Lambda$. %Under these assumptions, Lemma \ref{l4.1} guarantees the uniqueness of $\theta$ and $\ell$.
%Now let $\bm{\gamma}: I\times\Lambda\to \mathbb{R}^2_1$ be a one-parameter family of curves,
%and
Let $\bm{e}:U\rightarrow I\times\Lambda$, $\bm{e}(u)=(t(u),\lambda(u))$ be a smooth curve. We denote $\bm{E}(u)=\bm{\gamma}\circ\bm{e}(u)$ and $\bm{\Gamma}_\lambda(t)=\bm{\gamma}(t,\lambda).$ Note that $\bm{e}$ is not necessarily regular.

\begin{definition}\label{definition_envelope of mixed curve family}
{\rm
We call $\bm{E}$ an {\it envelope} (and $\bm{e}$ a {\it pre-envelope}) of the family of curves $\bm{\gamma}$ if the following conditions are satisfied:
\begin{enumerate}
\item[(1)] (Variability Condition) The function $\lambda$ is surjective and is not constant on any non-trivial subinterval of $U$.
\item[(2)] (Tangency Condition) The curve $\bm{E}$ is tangent at $u$ to the curve $\bm{\Gamma}_{\lambda(u)}$ at $t(u)$ for all $u\in U$, that is, $$\mbox{det}\big(\bm{E}'(u),\bm{\mu}(\bm{e}(u))\big)=0$$
    holds for all $u\in U$, where $\bm{\mu}(\bm{e}(u))=\cos\theta(\bm{e}(u))\mathbb{L}^++\sin\theta(\bm{e}(u))\mathbb{L}^-.$
\end{enumerate}
}
\end{definition}

The following theorem characterizes envelopes of one-parameter families of curves in the Lorentz-Minkowski plane. %in terms of the lightlike tangential data $(\alpha,\beta,m,n)$.

\begin{theorem}\label{T3}
Let $\bm{\gamma}: I\times\Lambda\to \mathbb{R}^2_1$ be a one-parameter family of curves, %with lightlike tangential data $(\alpha,\beta,m,n)$,
and let $\bm{e}:U\rightarrow I\times\Lambda$ be a smooth curve satisfying the variability condition. Then $\bm{e}$ is a pre-envelope of $\bm{\gamma}$ (and $\bm{E}=\bm{\gamma}\circ\bm{e}$ is an envelope of $\bm{\gamma}$) if and only if
$$n(\bm{e}(u))\cos\theta(\bm{e}(u))=m(\bm{e}(u))\sin\theta(\bm{e}(u))$$
holds for all $u\in U.$
\end{theorem}
\begin{proof}
Let $\bm{e}(u)=(t(u),\lambda(u))$ be a smooth curve satisfying the variability condition. Define $\bm{E}(u)=\bm{\gamma}(\bm{e}(u))$,
%and $\bm{\Gamma}_{\lambda(u)}(t(u))=\bm{\gamma}(t(u),\lambda(u)).$
then
\begin{align*}
\bm{E}'(u)=&\bm{\gamma}_t(\bm{e}(u))t'(u)+\bm{\gamma}_\lambda(\bm{e}(u))\lambda'(u)\\
=&\big(\alpha(\bm{e}(u))\mathbb{L}^++\beta(\bm{e}(u))\mathbb{L}^-\big)t'(u)
+\big(m(\bm{e}(u))\mathbb{L}^++n(\bm{e}(u))\mathbb{L}^-\big)\lambda'(u)\\
=&\left[\alpha(\bm{e}(u))t'(u)+m(\bm{e}(u))\lambda'(u)\right]\mathbb{L}^++
\left[\beta(\bm{e}(u))t'(u)+n(\bm{e}(u))\lambda'(u)\right]\mathbb{L}^-
\\
=&\left[\ell(\bm{e}(u))t'(u)\cos\theta(\bm{e}(u))+m(\bm{e}(u))\lambda'(u)\right]\mathbb{L}^+
+\left[\ell(\bm{e}(u))t'(u)\sin\theta(\bm{e}(u))+n(\bm{e}(u))\lambda'(u)\right]\mathbb{L}^-.
\end{align*}
%and
%\begin{align*}
%\frac{d\bm{\Gamma}_{\lambda(u)}}{dt}(t(u))=\bm{\gamma}_t(t(u),\lambda(u))=\alpha(\bm{e}(u))\mathbb{L}^++\beta(\bm{e}(u))\mathbb{L}^-.
%\end{align*}
The tangency condition $$\mbox{det}\big(\bm{E}'(u),\bm{\mu}(\bm{e}(u))\big)=0$$
holds for all $u\in U$ is equivalent to
$$(\ell(\bm{e}(u))t'(u)\sin\theta(\bm{e}(u))+n(\bm{e}(u))\lambda'(u))\cos\theta(\bm{e}(u))=(\ell(\bm{e}(u))t'(u)\cos\theta(\bm{e}(u))+m(\bm{e}(u))\lambda'(u))\sin\theta(\bm{e}(u))$$
holds for all $u\in U$.
Simplifying this equation yields
$$\lambda'(u)n(\bm{e}(u))\cos\theta(\bm{e}(u))=\lambda'(u)m(\bm{e}(u))\sin\theta(\bm{e}(u)).$$
Since $\bm{e}$ satisfies the variability condition, $\lambda'(u)\neq0$ on a dense subset of $U$. Therefore, the equality holds for all $u\in U$ if and only if
$$n(\bm{e}(u))\cos\theta(\bm{e}(u))=m(\bm{e}(u))\sin\theta(\bm{e}(u))$$
holds for all $u\in U$.
\end{proof}

Let's review Examples \ref{example astroid1} and \ref{example closed curves} by applying Theorem \ref{T3}.

\begin{example}\label{e4}
In Example \ref{example astroid1}, $\bm{\gamma}: [0,2\pi)\times[0,2\pi)\to \mathbb{R}^2_1$ is a one-parameter family of astroid curves defined by $$\bm{\gamma}(t,\lambda)=(2\cos\lambda\cos^3t-2\sin\lambda\sin^3t,2\sin\lambda\cos^3t+2\cos\lambda\sin^3t),$$
with lightlike tangential data
\begin{align*}
\alpha(t,\lambda)&=3\sin t\cos t\left[-(\cos\lambda+\sin\lambda)\cos t+(\cos\lambda-\sin\lambda)\sin t\right],\\
\beta(t,\lambda)&=3\sin t\cos t\left[(-\cos\lambda+\sin\lambda)\cos t-(\cos\lambda+\sin\lambda)\sin t\right],\\
m(t,\lambda)&=(\cos\lambda-\sin\lambda)\cos^3 t-(\cos\lambda+\sin\lambda)\sin^3 t,\\
n(t,\lambda)&=-(\cos\lambda+\sin\lambda)\cos^3 t-(\cos\lambda-\sin\lambda)\sin^3 t.
\end{align*}
There exist functions $\theta,\ell:I\times\Lambda\to \mathbb{R}$ defined by
\begin{align*}
\theta(t,\lambda)=t-\lambda+\frac{\pi}{4},\quad
\ell(t,\lambda)=-3\sqrt{2}\sin t\cos t
\end{align*}
such that $(\alpha(t,\lambda),\beta(t,\lambda))
=\ell(t,\lambda)(\cos\theta(t,\lambda),\sin\theta(t,\lambda))$.
Consider the curves
\begin{align*}
\bm{e}(u)=\bigg(\frac{\pi}{4},u\bigg),\quad
\bm{e}(u)=\bigg(\frac{3\pi}{4},u\bigg),\quad
\bm{e}(u)=\bigg(\frac{5\pi}{4},u\bigg),\quad
\bm{e}(u)=\bigg(\frac{7\pi}{4},u\bigg),
\end{align*}
each mapping $[0,2\pi)\rightarrow[0,2\pi)\times[0,2\pi).$  Each of these curves satisfies the variability condition. Moreover, one verifies that
\begin{align*}
n(\bm{e}(u))\cos\theta(\bm{e}(u))=m(\bm{e}(u))\sin\theta(\bm{e}(u)),
\end{align*}
so by Theorem \ref{T3}, each $\bm{e}$ is a pre-envelope of $\bm{\gamma}$. The corresponding envelopes are
\begin{align*}
\bm{E}(u)=&\left(\cos\bigg(u+\frac{\pi}{4}\bigg),\sin\bigg(u+\frac{\pi}{4}\bigg)\right),\quad
\left(\cos\bigg(u+\frac{3\pi}{4}\bigg),\sin\bigg(u+\frac{3\pi}{4}\bigg)\right),\\
&\left(\cos\bigg(u+\frac{5\pi}{4}\bigg),\sin\bigg(u+\frac{5\pi}{4}\bigg)\right),\quad
\left(\cos\bigg(u+\frac{7\pi}{4}\bigg),\sin\bigg(u+\frac{7\pi}{4}\bigg)\right),
\end{align*}
respectively (see Fig. \ref{figure_example4}). \\
If instead we take $\bm{e}(u)=(0,u)$, then
$$\bm{E}(u)=(2\cos u, 2\sin u),$$
but in this case $n(\bm{e}(u))\cos\theta(\bm{e}(u))\neq m(\bm{e}(u))\sin\theta(\bm{e}(u))$. Hence, by Theorem \ref{T3}, $\bm{E}(u)=(2\cos u, 2\sin u)$ is not an envelope of $\bm{\gamma}$.
Therefore, Theorem \ref{T3} reveals that the set $\mathcal{D}$ calculated in Example \ref{example astroid} is precisely the union of the Euclidean circle $(2\cos u, 2\sin u)$ and the envelope of $\bm{\gamma}$.
\begin{figure}[h]
\begin{center}
\includegraphics[width=6.2cm]
{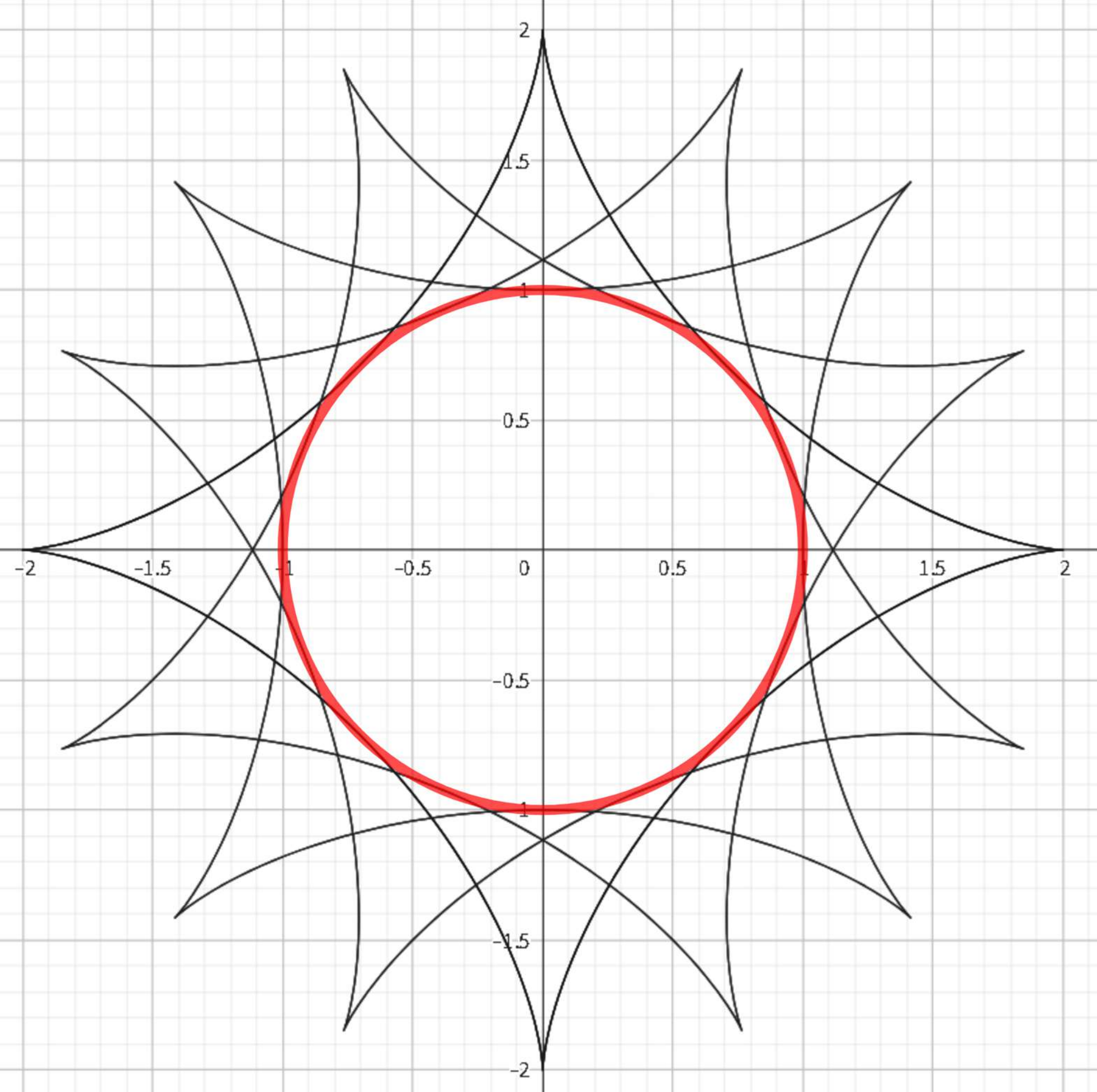}

\caption{One-parameter family of astroid curves and its envelope.
}
\label{figure_example4}
\end{center}
\end{figure}
\end{example}

\begin{example}
In Example \ref{example closed curves}, $\bm{\gamma}: [0,2\pi)\times[0,2\pi)\to \mathbb{R}^2_1$ is a one-parameter family of closed curves defined by
\begin{align*}
\bm{\gamma}(t,\lambda)=\left(\frac{1}{3}\cos\lambda\cos^3t+\frac{1}{3}\sin\lambda\sin^3t-\sin\lambda\sin t,~ \frac{1}{3}\sin\lambda\cos^3t-\frac{1}{3}\cos\lambda\sin^3t+\cos\lambda\sin t\right)
\end{align*}
with lightlike tangential data
\begin{align*}
\alpha(t,\lambda)&=-\frac{1}{2}\cos^2t\left[\cos\lambda(\sin t-\cos t)+\sin\lambda(\sin t+\cos t)\right],\\
\beta(t,\lambda)&=\frac{1}{2}\cos^2t\left[\sin\lambda(\sin t-\cos t)-\cos\lambda(\sin t+\cos t)\right],\\
m(t,\lambda)&=\frac{1}{6}(\cos\lambda-\sin\lambda)\cos^3t-\frac{1}{2}(\cos\lambda+\sin\lambda)\sin t+\frac{1}{6}(\cos\lambda+\sin\lambda)\sin^3t,\\
n(t,\lambda)&=-\frac{1}{6}(\cos\lambda+\sin\lambda)\cos^3t-\frac{1}{2}(\cos\lambda-\sin\lambda)\sin t+\frac{1}{6}(\cos\lambda-\sin\lambda)\sin^3t.
\end{align*}
There exist functions $\theta,\ell:I\times\Lambda\to \mathbb{R}$ defined by
\begin{align*}
\theta(t,\lambda)=-t-\lambda-\frac{\pi}{4},\quad
\ell(t,\lambda)=\frac{\sqrt{2}}{2}\cos^2t
\end{align*}
such that $(\alpha(t,\lambda),\beta(t,\lambda))
=\ell(t,\lambda)(\cos\theta(t,\lambda),\sin\theta(t,\lambda))$.
Consider the curves
\begin{align*}
\bm{e}(u)=(0,u),\quad
\bm{e}(u)=\bigg(\frac{\pi}{2},u\bigg),\quad
\bm{e}(u)=(\pi,u),\quad
\bm{e}(u)=\bigg(\frac{3\pi}{2},u\bigg),
\end{align*}
each mapping $[0,2\pi)\rightarrow[0,2\pi)\times[0,2\pi).$  Each of these curves satisfies the variability condition. Moreover, one verifies that
\begin{align*}
n(\bm{e}(u))\cos\theta(\bm{e}(u))=m(\bm{e}(u))\sin\theta(\bm{e}(u)),
\end{align*}
so by Theorem \ref{T3}, each $\bm{e}$ is a pre-envelope of $\bm{\gamma}$. The corresponding envelopes are
\begin{align*}
\bm{E}(u)=&\left(\frac{1}{3}\cos u,\frac{1}{3}\sin u\right),\quad
\left(-\frac{2}{3}\sin u,\frac{2}{3}\cos u\right),\\
&\left(-\frac{1}{3}\cos u,-\frac{1}{3}\sin u\right),\quad
\left(\frac{2}{3}\sin u,-\frac{2}{3}\cos u\right),
\end{align*}
respectively (see Fig \ref{figure_example5}). In contrast to Example \ref{e4},  the difference here is that the curve traced by the singular points of each member of the family now forms an envelope.
\begin{figure}[h]
\begin{center}
\includegraphics[width=6.6cm]
{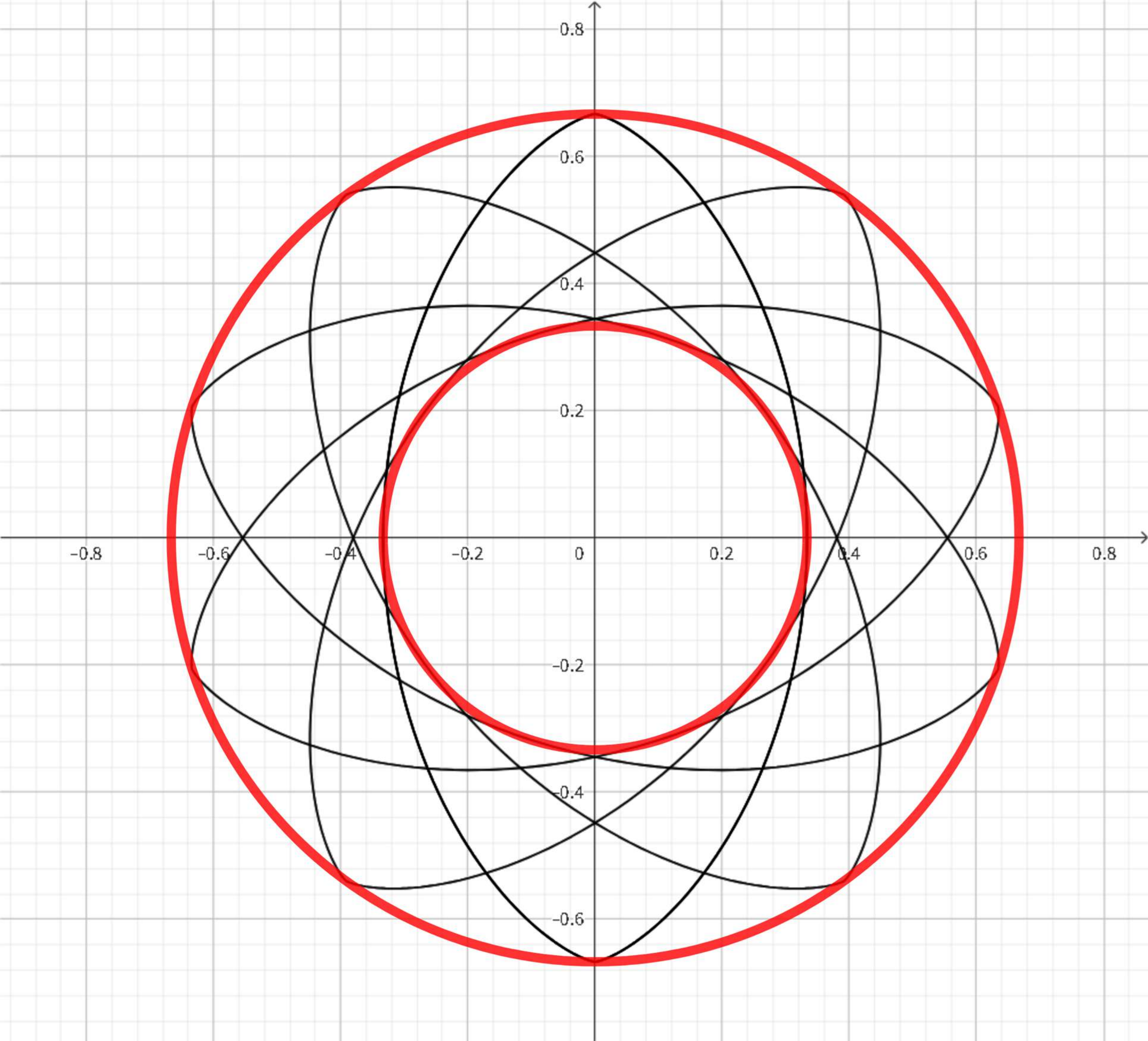}

\caption{One-parameter family of closed curves and its envelopes.
}
\label{figure_example5}
\end{center}
\end{figure}
\end{example}

\begin{definition}
A map $\Phi:\widetilde{I}\times\widetilde{\Lambda}\rightarrow I\times\Lambda$ is called a {\it one-parameter family of parameter change} if $\Phi$ is a diffeomorphism and can be expressed in the form
$$\Phi(s,k)=(\phi(s,k),\; \varphi(k)),$$
where $\phi$ and $\varphi$ are smooth functions.
\end{definition}

\begin{proposition}
Let $\bm{\gamma}: I\times\Lambda\to \mathbb{R}^2_1$ be a one-parameter family of curves with the lightlike tangential data $(\alpha,\beta,m,n)$. Suppose that $\Phi:\widetilde{I}\times\widetilde{\Lambda}\rightarrow I\times\Lambda$ is a one-parameter family of parameter change, $\bm{e}:U\rightarrow I\times\Lambda$ is a pre-envelope, and $\bm{E}=\bm{\gamma}\circ\bm{e}$ is an envelope. Define $\widetilde{\bm{\gamma}}=\bm{\gamma}\circ\Phi:\widetilde{I}\times\widetilde{\Lambda}\rightarrow\mathbb{R}^2_1$.
Then $\widetilde{\bm{\gamma}}$ is a one-parameter family of curves with lightlike tangential data given by
\begin{align*}
\widetilde{\alpha}(s,k)&=\phi_s(s,k)\alpha(\phi(s,k),\varphi(k)),\\
\widetilde{\beta}(s,k)&=\phi_s(s,k)\beta(\phi(s,k),\varphi(k)),\\
\widetilde{m}(s,k)&=\phi_k(s,k)\alpha(\phi(s,k),\varphi(k))+\varphi'(k)m(\phi(s,k),\varphi(k)),\\
\widetilde{n}(s,k)&=\phi_k(s,k)\beta(\phi(s,k),\varphi(k))+\varphi'(k)n(\phi(s,k),\varphi(k)).
\end{align*}
Moreover, $\Phi^{-1}\circ\bm{e}:U\rightarrow\widetilde{I}\times\widetilde{\Lambda}$ is a pre-envelope of $\widetilde{\bm{\gamma}}$, and $\bm{E}$ is also an envelope of $\widetilde{\bm{\gamma}}$.
\end{proposition}
\begin{proof}
Since $\Phi(s,k)=(\phi(s,k),\varphi(k)),$ we have $\widetilde{\bm{\gamma}}(s,k)=\bm{\gamma}(\phi(s,k),\varphi(k))$. Differentiating, \begin{align*}
\widetilde{\bm{\gamma}}_s(s,k)&=\phi_s(s,k)\bm{\gamma}_t(\phi(s,k),\varphi(k)),\\ \widetilde{\bm{\gamma}}_k(s,k)&=\phi_k(s,k)\bm{\gamma}_t(\phi(s,k),\varphi(k))+\varphi'(k)\bm{\gamma}_k(\phi(s,k),\varphi(k)).
\end{align*}
A direct calculation then yields the lightlike tangential data of $\widetilde{\bm{\gamma}}$ as following:
\begin{align*}
\widetilde{\alpha}(s,k)&=\phi_s(s,k)\alpha(\phi(s,k),\varphi(k)),\\
\widetilde{\beta}(s,k)&=\phi_s(s,k)\beta(\phi(s,k),\varphi(k)),\\
\widetilde{m}(s,k)&=\phi_k(s,k)\alpha(\phi(s,k),\varphi(k))+\varphi'(k)m(\phi(s,k),\varphi(k)),\\
\widetilde{n}(s,k)&=\phi_k(s,k)\beta(\phi(s,k),\varphi(k))+\varphi'(k)n(\phi(s,k),\varphi(k)).
\end{align*}
The inverse map $\Phi^{-1}:I\times\Lambda\rightarrow\widetilde{I}\times\widetilde{\Lambda}$  can be written as  $\Phi^{-1}(t,\lambda)=(\psi(t,\lambda),\varphi^{-1}(\lambda))$. Consequently, $$\Phi^{-1}\circ\bm{e}(u)=(\psi(t(u),\lambda(u)),\varphi^{-1}(\lambda(u))).$$
Because $\bm{e}$ is a pre-envelope of $\bm{\gamma}$ and $(d/du)\varphi^{-1}(\lambda(u))=\lambda'(u)\varphi^{-1}_\lambda(\lambda(u))$,  the variability condition holds. Moreover, using the relation
\begin{align*}
\widetilde{\bm{\gamma}}_s(\Phi^{-1}\circ\bm{e}(u))=&\phi_s(\psi(t(u),\lambda(u)),\varphi^{-1}(\lambda(u)))\alpha(\bm{e}(u))\mathbb{L}^+
+\phi_s(\psi(t(u),\lambda(u)),\varphi^{-1}(\lambda(u)))\beta(\bm{e}(u))\mathbb{L}^-,\\
=&\phi_s(\psi(t(u),\lambda(u)),\varphi^{-1}(\lambda(u)))\ell(\bm{e}(u))\cos\theta(\bm{e}(u))\mathbb{L}^+\\
&+\phi_s(\psi(t(u),\lambda(u)),\varphi^{-1}(\lambda(u)))\ell(\bm{e}(u))\sin\theta(\bm{e}(u))\mathbb{L}^-,
\end{align*}
we have
\begin{align*}
\widetilde{\theta}(\Phi^{-1}\circ\bm{e}(u))=\theta(\bm{e}(u)),\quad
\widetilde{\ell}(\Phi^{-1}\circ\bm{e}(u))=\phi_s(\psi(t(u),\lambda(u)),\varphi^{-1}(\lambda(u)))\ell(\bm{e}(u)).
\end{align*}
By the relation
\begin{align*}
\widetilde{\bm{\gamma}}_k(\Phi^{-1}\circ\bm{e}(u))=&\left[\phi_k(\psi(t(u),\lambda(u)),\varphi^{-1}(\lambda(u)))\alpha(\bm{e}(u))
+\varphi_k(\varphi^{-1}(\lambda(u)))m(\bm{e}(u))\right]\mathbb{L}^+\\
&+\left[\phi_k(\psi(t(u),\lambda(u)),\varphi^{-1}(\lambda(u)))\beta(\bm{e}(u))
+\varphi_k(\varphi^{-1}(\lambda(u)))n(\bm{e}(u))\right]\mathbb{L}^-
\end{align*}
and the condition $n(\bm{e}(u))\cos\theta(\bm{e}(u))=m(\bm{e}(u))\sin\theta(\bm{e}(u))$, we have
\begin{align*}
&\widetilde{n}(\Phi^{-1}\circ\bm{e}(u))\cos\widetilde{\theta}(\Phi^{-1}\circ\bm{e}(u))\\
=&
\big[\phi_k(\psi(t(u),\lambda(u)),\varphi^{-1}(\lambda(u)))\beta(\bm{e}(u))
+\varphi_k(\varphi^{-1}(\lambda(u)))n(\bm{e}(u))\big]\cos\theta(\bm{e}(u))\\
=&\phi_k(\psi(t(u),\lambda(u)),\varphi^{-1}(\lambda(u)))\beta(\bm{e}(u))\cos\theta(\bm{e}(u))
+\varphi_k(\varphi^{-1}(\lambda(u)))n(\bm{e}(u))\cos\theta(\bm{e}(u))\\
=&\phi_k(\psi(t(u),\lambda(u)),\varphi^{-1}(\lambda(u)))\ell(\bm{e}(u))\cos\theta(\bm{e}(u))\sin\theta(\bm{e}(u))
+\varphi_k(\varphi^{-1}(\lambda(u)))m(\bm{e}(u))\sin\theta(\bm{e}(u))\\
=&\big[\phi_k(\psi(t(u),\lambda(u)),\varphi^{-1}(\lambda(u)))\alpha(\bm{e}(u))
+\varphi_k(\varphi^{-1}(\lambda(u)))m(\bm{e}(u))
\big]\sin\theta(\bm{e}(u))\\
=&\widetilde{m}(\Phi^{-1}\circ\bm{e}(u))\sin\widetilde{\theta}(\Phi^{-1}\circ\bm{e}(u)).
\end{align*}
Thus, by Theorem \ref{T3}, $\Phi^{-1}\circ\bm{e}$ is a pre-envelope of $\widetilde{\bm{\gamma}}$. Finally, since $$\widetilde{\bm{\gamma}}(\Phi^{-1}\circ\bm{e}(u))=\bm{\gamma}\circ\Phi\circ\Phi^{-1}\circ\bm{e}(u)
=\bm{\gamma}\circ\bm{e}(u)=\bm{E}(u),$$
it follows that $\bm{E}$ is also an envelope of $\widetilde{\bm{\gamma}}$.
\end{proof}

In Examples \ref{example astroid} and \ref{e4}, the $\mathcal{D}$ envelope (Definition \ref{classical definition}) turns out to be strictly larger than the $\bm{E}$ envelope (Definition \ref{definition_envelope of mixed curve family}). Namely, the Euclidean circle $\{(x,y)\in\mathbb{R}^2_1|x^2+y^2=4\}$ is redundant. Therefore, it is both interesting and meaningful to investigate the precise relationship between $\mathcal{D}$ envelopes and $\bm{E}$ envelopes for one-parameter families of curves in the Lorentz-Minkowski
plane.

\begin{theorem}\label{T4}
Let $\bm{\gamma}: I\times\Lambda\to \mathbb{R}^2_1$ be a one-parameter family of curves defined implicitly by $F(x,y,\lambda)=0$, that is, $F(x(t,\lambda),y(t,\lambda),\lambda)=0,$
where $\bm{\gamma}(t,\lambda)=(x(t,\lambda),y(t,\lambda))$. Let $\bm{e}:U\rightarrow I\times\Lambda$ be a pre-envelope of $\bm{\gamma}$ (i.e., $\bm{E}=\bm{\gamma}\circ\bm{e}:U\rightarrow\mathbb{R}^2_1$ is an envelope of $\bm{\gamma}$).
%If the set
%$$\{u\in U\mid(\alpha(\bm{e}(u)),\beta(\bm{e}(u)))\neq(0,0)\}$$
%is dense in $U$,
Then $\bm{E}\subset\mathcal{D}$.
\end{theorem}
\begin{proof}
Differentiating the identity $F(x(t,\lambda),y(t,\lambda),\lambda)=0$
 with respect to $t$ and $\lambda$ gives
\begin{align*}
&F_x(x(t,\lambda),y(t,\lambda),\lambda)x_t(t,\lambda)+F_y(x(t,\lambda),y(t,\lambda),\lambda)y_t(t,\lambda)=0,\\
&F_x(x(t,\lambda),y(t,\lambda),\lambda)x_\lambda(t,\lambda)+F_y(x(t,\lambda),y(t,\lambda),\lambda)y_\lambda(t,\lambda)
+F_\lambda(x(t,\lambda),y(t,\lambda),\lambda)=0.
\end{align*}
There exist functions $q(t,\lambda)$ and $p(t,\lambda)$ such that
\begin{align*}
(-F_x(x(t,\lambda),y(t,\lambda),\lambda),F_y(x(t,\lambda),y(t,\lambda),\lambda))=q(t,\lambda)\mathbb{L}^++p(t,\lambda)\mathbb{L}^-.
\end{align*}
Recall that
\begin{align*}
\bm{\gamma}_t(t,\lambda)&=(x_t(t,\lambda),y_t(t,\lambda))=\alpha(t,\lambda)\mathbb{L}^++\beta(t,\lambda)\mathbb{L}^-,\\
\bm{\gamma}_\lambda(t,\lambda)&=(x_\lambda(t,\lambda),y_\lambda(t,\lambda))=m(t,\lambda)\mathbb{L}^++n(t,\lambda)\mathbb{L}^-.\\
\end{align*}
A direct computation yields
\begin{align*}
&\big\langle\big(-F_x(x(t,\lambda),y(t,\lambda),\lambda),F_y(x(t,\lambda),y(t,\lambda),\lambda)\big),\big(x_t(t,\lambda),y_t(t,\lambda)\big)\big\rangle\\
=&-2(q(t,\lambda)\beta(t,\lambda)+p(t,\lambda)\alpha(t,\lambda))\\
=&0,
\end{align*}
hence
$$q(t,\lambda)\beta(t,\lambda)+p(t,\lambda)\alpha(t,\lambda)
=\ell(t,\lambda)\big(q(t,\lambda)\sin\theta(t,\lambda)+p(t,\lambda)\cos\theta(t,\lambda)\big)=0$$
for all $(t,\lambda)\in I\times\Lambda$.
Note that the set $\overline{\Sigma}=\{(t,\lambda)\in I\times\Lambda\mid\ell(t,\lambda)\neq0\}$ is dense in $I\times\Lambda$, it follows that
$$q(t,\lambda)\sin\theta(t,\lambda)+p(t,\lambda)\cos\theta(t,\lambda)=0$$
for all $(t,\lambda)\in I\times\Lambda$.
Now let $\bm{e}:U\rightarrow I\times\Lambda,~ \bm{e}(u)=(t(u),\lambda(u))$, be a pre-envelope of $\bm{\gamma}$. By Theorem \ref{T3},
$$n(\bm{e}(u))\cos\theta(\bm{e}(u))=m(\bm{e}(u))\sin\theta(\bm{e}(u))$$
holds for all $u\in U.$ Thus, we obtain
\begin{align*}
\left(\begin{array}{cc}
        q(\bm{e}(u)) & p(\bm{e}(u)) \\
        -m(\bm{e}(u)) & n(\bm{e}(u))
      \end{array}
\right)
\left(\begin{array}{c}
        \sin\theta(\bm{e}(u)) \\
       \cos\theta(\bm{e}(u))
      \end{array}
\right)=
\left(\begin{array}{c}
        0 \\
        0
      \end{array}
\right)
\end{align*}
holds for all $u\in U.$
%By the assumption the set $\{u\in U|(\alpha(\bm{e}(u)),\beta(\bm{e}(u)))\neq(0,0)\}$ is dense in $U$,
It follows that
\begin{align*}
q(\bm{e}(u))n(\bm{e}(u))+p(\bm{e}(u))m(\bm{e}(u))=0
\end{align*}
always holds for $u\in U.$ Then, we have
\begin{align*}
&F_x(x(t(u),\lambda(u)),y(t(u),\lambda(u)),\lambda(u))x_\lambda(t(u),\lambda(u))
+F_y(x(t(u),\lambda(u)),y(t(u),\lambda(u)),\lambda(u))y_\lambda(t(u),\lambda(u))\\
=&\langle(q(\bm{e}(u))\mathbb{L}^++p(\bm{e}(u))\mathbb{L}^-),(m(\bm{e}(u))\mathbb{L}^++n(\bm{e}(u))\mathbb{L}^-)\rangle\\
=&-2\left(q(\bm{e}(u))n(\bm{e}(u))+p(\bm{e}(u))m(\bm{e}(u))\right)\\
=&0.
\end{align*}
From the differentiated equation with respect to $\lambda$, this implies $$F_\lambda(x(t(u),\lambda(u)),y(t(u),\lambda(u)),\lambda(u))=0$$ for all $u\in U.$ Therefore, we have $\bm{E}(u)\in\mathcal{D}$ for all $u\in U$.
\end{proof}

To establish the converse, the following proposition is required.

\begin{proposition}\label{P4}
Let $\bm{\gamma}: I\times\Lambda\to \mathbb{R}^2_1$ be a one-parameter family of curves, and let $\bm{e}:U\rightarrow I\times\Lambda$ be a smooth curve satisfying the variability condition. Suppose the set of non-zero
points of $\ell\circ\bm{e}$ is dense in $U$ and the trace of $\bm{e}$ lies in the singular set of $\bm{\gamma}$, then $\bm{e}$ is a pre-envelope of $\bm{\gamma}$ (i.e., $\bm{E}=\bm{\gamma}\circ\bm{e}$ is an envelope of $\bm{\gamma}$).
\end{proposition}
\begin{proof}
Since $\bm{e}(u)$ belongs to the singular set of $\bm{\gamma}$, we have
\begin{align*}
\mathrm{det}(\bm{\gamma}_t(\bm{e}(u)),\bm{\gamma}_\lambda(\bm{e}(u)))=0~~~~~~ \mbox{for all $u\in U$.}
\end{align*}
Thus $$\bm{\gamma}_t(\bm{e}(u))=\alpha(\bm{e}(u))\mathbb{L}^++\beta(\bm{e}(u))\mathbb{L}^-=\ell(\bm{e}(u))\bm{\mu}(\bm{e}(u))$$ and $\bm{\gamma}_\lambda(\bm{e}(u))$
are linearly dependent.  By
the assumption and continuous property, it follows $\bm{\gamma}_\lambda(\bm{e}(u))$ and $\bm{\mu}(\bm{e}(u))$
are linearly dependent. Using the expressions
\begin{align*}
\bm{\mu}(\bm{e}(u))=\cos\theta(\bm{e}(u))\mathbb{L}^++\sin\theta(\bm{e}(u))\mathbb{L}^-,\quad
\bm{\gamma}_\lambda(\bm{e}(u))=m(\bm{e}(u))\mathbb{L}^++n(\bm{e}(u))\mathbb{L}^-,
\end{align*}
linear dependence implies
$$n(\bm{e}(u))\cos\theta(\bm{e}(u))=m(\bm{e}(u))\sin\theta(\bm{e}(u))$$
for all $u\in U$. By Theorem \ref{T3}, $\bm{e}$ is a pre-envelope of $\bm{\gamma}$ and consequently $\bm{E}=\bm{\gamma}\circ\bm{e}$ is an envelope of $\bm{\gamma}$.
\end{proof}

\begin{theorem}
Let $\bm{\gamma}: I\times\Lambda\to \mathbb{R}^2_1$ be a one-parameter family of curves defined implicitly by $F(x,y,\lambda)=0$, that is, $F(x(t,\lambda),y(t,\lambda),\lambda)=0,$
where $\bm{\gamma}(t,\lambda)=(x(t,\lambda),y(t,\lambda))$. Let $\bm{e}:U\rightarrow I\times\Lambda$ be a smooth curve satisfying the variability
condition. Suppose that $\bm{E}(u)=\bm{\gamma}\circ\bm{e}(u)\in\mathcal{D}$, the set of non-zero
points of $\ell\circ\bm{e}$ is dense in $U$, and the set
\begin{align*}
\{u\in U\mid(F_x,F_y)(x(t(u),\lambda(u)),y(t(u),\lambda(u)),\lambda(u))\neq(0,0)\}
\end{align*}
is also dense in $U$, where $\bm{e}(u)=(t(u),\lambda(u))$. Then $\bm{e}$ is a pre-envelope of $\bm{\gamma}$ (i.e., $\bm{E}=\bm{\gamma}\circ\bm{e}$ is an envelope of $\bm{\gamma}$).
\end{theorem}
\begin{proof}
Since $F(x,y,\lambda)=0$ is the implicit equation defining $\bm{\gamma}$, we have
\begin{align*}
&F_x(x(t,\lambda),y(t,\lambda),\lambda)x_t(t,\lambda)+F_y(x(t,\lambda),y(t,\lambda),\lambda)y_t(t,\lambda)=0,\\
&F_x(x(t,\lambda),y(t,\lambda),\lambda)x_\lambda(t,\lambda)+F_y(x(t,\lambda),y(t,\lambda),\lambda)y_\lambda(t,\lambda)
+F_\lambda(x(t,\lambda),y(t,\lambda),\lambda)=0.
\end{align*}
Since
$
\bm{E}(u)=\bm{\gamma}\circ\bm{e}(u)\in\mathcal{D},
$
it follows
\begin{align*}
F_\lambda(x(t(u),\lambda(u)),y(t(u),\lambda(u)),\lambda(u))=0
\end{align*}
for all $u\in U$ from Definition \ref{classical definition}. Thus, we have
\begin{align*}
\left(\begin{array}{cc}
        x_t(t(u),\lambda(u)) & y_t(t(u),\lambda(u)) \\
        x_\lambda(t(u),\lambda(u)) & y_\lambda(t(u),\lambda(u))
      \end{array}
\right)
\left(\begin{array}{c}
        F_x(x(t(u),\lambda(u)),y(t(u),\lambda(u)),\lambda(u)) \\
       F_y(x(t(u),\lambda(u)),y(t(u),\lambda(u)),\lambda(u))
      \end{array}
\right)=
\left(\begin{array}{c}
        0 \\
        0
      \end{array}
\right)
\end{align*}
for all $u\in U$.
Because the set
\begin{align*}
\{u\in U\mid(F_x,F_y)(x(t(u),\lambda(u)),y(t(u),\lambda(u)),\lambda(u))\neq(0,0)\}
\end{align*}
is dense in $U$, it follows that
\begin{align*}
\mathrm{det}(\bm{\gamma}_t(\bm{e}(u)),\bm{\gamma}_\lambda(\bm{e}(u)))=0
\end{align*}
for all $u\in U$. Therefore, the trace of $\bm{e}$ lies in the singular set of $\bm{\gamma}$. Since the set of non-zero
points of $\ell\circ\bm{e}$ is dense in $U$, by Proposition \ref{P4}, $\bm{e}$ is a pre-envelope of $\bm{\gamma}$, and consequently $\bm{E}$ is an envelope of $\bm{\gamma}$.
\end{proof}

\section{Contact at singular points}\label{Section5}
It is known that an envelope can also be characterized via contact. We therefore introduce the notion of contact between curves in the Lorentz-Minkowski plane.

\begin{definition}\label{D5}
Let $g:I\rightarrow\mathbb{R}^2_1$ be a curve, and let $\mathfrak{F}:\{(x,y)\in\mathbb{R}^2_1|f(x,y)=0\}$ be a curve given implicitly. Then $g$ is said to have {\it $(k+1)$-point contact} with $\mathfrak{F}$ at $u=u_0$ if
\begin{align*}
\frac{d^i}{du^i}f(g(u))=0~~(i=0,\ldots,k),\quad \frac{d^{k+1}}{du^{k+1}}f(g(u))\neq0.
\end{align*}
Moreover, $g$ is said to have {\it at least $(k+1)$-point contact} with $\mathfrak{F}$ at $u=u_0$ if
\begin{align*}
\frac{d^i}{du^i}f(g(u))=0~~(i=0,\ldots,k).
\end{align*}
\end{definition}

\begin{example}\label{E6}
Let $\bm{\gamma}:\mathbb{R}\times\mathbb{R}\rightarrow\mathbb{R}^2_1$ be a one-parameter family of lines defined by
$$\bm{\gamma}(t,\lambda)=(2t+\lambda^2,3t\lambda+\lambda^3).$$
Since
\begin{align*}
\bm{\gamma}_t(t,\lambda)&=(2,3\lambda)=\alpha(t,\lambda)\mathbb{L}^++\beta(t,\lambda)\mathbb{L}^-,\\
\bm{\gamma}_\lambda(t,\lambda)&=(2\lambda,3\lambda^2+3t)=m(t,\lambda)\mathbb{L}^++n(t,\lambda)\mathbb{L}^-,
\end{align*}
the lightlike tangential data of $\bm{\gamma}$ are
\begin{align*}
\alpha(t,\lambda)&=1+\frac{3}{2}\lambda,~~\beta(t,\lambda)=1-\frac{3}{2}\lambda,\\
m(t,\lambda)&=\frac{3}{2}\lambda^2+\lambda+\frac{3}{2}t,~~n(t,\lambda)=-\frac{3}{2}\lambda^2+\lambda-\frac{3}{2}t.
\end{align*}
There exist functions $\theta,\ell:I\times\Lambda\to \mathbb{R}$ such that
\begin{align*}
\cos\theta(t,\lambda)=\frac{\sqrt{2}(2+3\lambda)}{2\sqrt{4+9\lambda^2}},\quad
\sin\theta(t,\lambda)=\frac{\sqrt{2}(2-3\lambda)}{2\sqrt{4+9\lambda^2}},\quad
\ell(t,\lambda)=\frac{\sqrt{8 + 18\lambda^2}}{2},
\end{align*}
which satisfy the relation $(\alpha(t,\lambda),\beta(t,\lambda))
=\ell(t,\lambda)(\cos\theta(t,\lambda),\sin\theta(t,\lambda))$.

Consider the curve $\bm{e}:\mathbb{R}\rightarrow\mathbb{R}\times\mathbb{R}$ defined by $\bm{e}(u)=(0,u)$.  This curve satisfies the variability condition. Moreover, one verifies that
\begin{align*}
n(\bm{e}(u))\cos\theta(\bm{e}(u))=m(\bm{e}(u))\sin\theta(\bm{e}(u))=\frac{\sqrt{2}u(4-9u^2)}{4\sqrt{4+9u^2}},
\end{align*}
so by Theorem \ref{T3}, $\bm{e}$ is a pre-envelope of $\bm{\gamma}$. The corresponding envelope $\bm{E}$ is given by
\begin{align*}
\bm{E}(u)=(u^2,u^3),
\end{align*}
which has a singular point at $u=0$.
Each line in the family can be expressed implicitly as
$$\Gamma_{\lambda_0}=\{(x,y)\in\mathbb{R}^2_1|f_{\lambda_0}(x,y)=-3\lambda_0x+2y+\lambda_0^3=0\}.$$
To study the contact between $\bm{E}$ and $\Gamma_{\lambda_0}$, consider
\begin{align*}
f_{\lambda_0}(\bm{E}(u))=-3\lambda_0u^2+2u^3+\lambda_0^3.
\end{align*}
It is easy to see that $$f_{\lambda_0}(\bm{E}(u))=\frac{d}{du}f_{\lambda_0}(\bm{E}(u))=0$$
at $u=\lambda_0$. Hence, by Definition \ref{D5}, $\bm{E}$ and $\Gamma_{\lambda_0}$ have at least $2$-point contact at $u=\lambda_0$. \\
$\bullet$ If $u=\lambda_0\neq0$, then
\begin{align*}
f_{\lambda_0}(\bm{E}(u))=\frac{d}{du}f_{\lambda_0}(\bm{E}(u))=0, \quad \frac{d^2}{du^2}f_{\lambda_0}(\bm{E}(u))\neq0,
\end{align*}
so they have exactly 2-point contact.\\
$\bullet$ If $u=\lambda_0=0$, then
\begin{align*}
f_{\lambda_0}(\bm{E}(u))=\frac{d}{du}f_{\lambda_0}(\bm{E}(u))=\frac{d^2}{du^2}f_{\lambda_0}(\bm{E}(u))=0,
\quad \frac{d^3}{du^3}f_{\lambda_0}(\bm{E}(u))\neq0.
\end{align*}
so they have 3-point contact at the singular point of $\bm{E}$.
\end{example}
\begin{problem}\label{p2}
In Example \ref{E6}, the envelope $\bm{E}$ and the curve $\Gamma_{0}$ exhibit higher-order contact at the singular point of $\bm{E}$. This motivates a natural question: In the Lorentz-Minkowski plane, does higher-order contact always occur between an envelope of a one-parameter curve family and a member curve of the family at any singular point of the envelope?
\end{problem}

The following theorem gives the answer to Problem \ref{p2}.

\begin{theorem}
Let $\bm{\gamma}: I\times\Lambda\to \mathbb{R}^2_1$ be a one-parameter family of  curves defined implicitly by $F(x,y,\lambda)=0$, that is, $F(x(t,\lambda),y(t,\lambda),\lambda)=0,$
where $\bm{\gamma}(t,\lambda)=(x(t,\lambda),y(t,\lambda))$. Let $\bm{e}:U\rightarrow I\times\Lambda$ be a pre-envelope of $\bm{\gamma}$ (i.e., $\bm{E}=\bm{\gamma}\circ\bm{e}$ is an envelope of $\bm{\gamma}$). Suppose that the set
of non-zero
points of $\ell\circ\bm{e}$ is dense in $U$. % and that the set
%\begin{align*}
%\{u\in U\mid(F_x,F_y)(x(t(u),\lambda(u)),y(t(u),\lambda(u)),\lambda(u))\neq(0,0)\}
%\end{align*}
%is also dense in $U$, where $\bm{e}(u)=(t(u),\lambda(u))$.
If $u=u_0$ is a singular point of $\bm{E}$ and $\ell\circ\bm{e}(u_0)\neq0$, %and
%$$\frac{d}{du}(\alpha(\bm{e}(u))n(\bm{e}(u)))=0$$
%at $u=u_0$,
then $\bm{E}$ and $\Gamma_{\lambda_0}$ have at least $3$-point contact at $u=u_0$.
\end{theorem}
\begin{proof}
Let $ \bm{e}(u)=(t(u),\lambda(u))$ be a pre-envelope of $\bm{\gamma}$, and $f_{\lambda(u_0)}=F(x,y,\lambda(u_0))=0$ the implicit equation defining $\Gamma_{\lambda_0}$. Then
\begin{align*}
f_{\lambda(u_0)}(\bm{E}(u))=F(x(t(u),\lambda(u)),y(t(u),\lambda(u)),\lambda(u_0)),
\end{align*}
where $\bm{E}(u)=\bm{\gamma}(t(u),\lambda(u))=(x(t(u),\lambda(u)),y(t(u),\lambda(u)))$.
Since
\begin{align*}
F(x(t(u),\lambda(u)),y(t(u),\lambda(u)),\lambda(u))=0
\end{align*}
for all $u\in U,$ it follows that
\begin{align*}
f_{\lambda(u_0)}(\bm{E}(u_0))=F(x(t(u_0),\lambda(u_0)),y(t(u_0),\lambda(u_0)),\lambda(u_0))=0.
\end{align*}
Differentiating $F(x(t,\lambda),y(t,\lambda),\lambda)=0$ with respect to $t$ and $\lambda$ gives
\begin{align*}
&F_x(x(t,\lambda),y(t,\lambda),\lambda)x_t(t,\lambda)+F_y(x(t,\lambda),y(t,\lambda),\lambda)y_t(t,\lambda)=0,\\
&F_x(x(t,\lambda),y(t,\lambda),\lambda)x_\lambda(t,\lambda)+F_y(x(t,\lambda),y(t,\lambda),\lambda)y_\lambda(t,\lambda)
+F_\lambda(x(t,\lambda),y(t,\lambda),\lambda)=0.
\end{align*}
Because
%the set $\{u\in U|(\alpha(\bm{e}(u)),\beta(\bm{e}(u)))\neq(0,0)\}$ is dense in $U$ and
$\bm{e}$ is a pre-envelope of $\bm{\gamma}$, it follows
\begin{align*}
&F_x(x(t(u),\lambda(u)),y(t(u),\lambda(u)),\lambda(u))x_t(t(u),\lambda(u))
+F_y(x(t(u),\lambda(u)),y(t(u),\lambda(u)),\lambda(u))y_t(t(u),\lambda(u))=0,\\
&F_x(x(t(u),\lambda(u)),y(t(u),\lambda(u)),\lambda(u))x_\lambda(t(u),\lambda(u))
+F_y(x(t(u),\lambda(u)),y(t(u),\lambda(u)),\lambda(u))y_\lambda(t(u),\lambda(u))=0
\end{align*}
for all $u\in U$ from Theorem \ref{T4}. Consequently,
\begin{align*}
\frac{d}{du}f_{\lambda(u_0)}(\bm{E}(u))=&\frac{d}{du}F(x(t(u),\lambda(u)),y(t(u),\lambda(u)),\lambda(u_0))\\
=&[F_x(x(t(u),\lambda(u)),y(t(u),\lambda(u)),\lambda(u_0))x_t(t(u),\lambda(u))\\&+F_y(x(t(u),\lambda(u)),y(t(u),\lambda(u)),\lambda(u_0))y_t(t(u),\lambda(u))]t'(u)\\
&+[F_x(x(t(u),\lambda(u)),y(t(u),\lambda(u)),\lambda(u_0))x_\lambda(t(u),\lambda(u))\\&+F_y(x(t(u),\lambda(u)),y(t(u),\lambda(u)),\lambda(u_0))y_\lambda(t(u),\lambda(u))]
\lambda'(u)\\
=&0
\end{align*}
for $u=u_0.$ Hence
\begin{align*}
f_{\lambda(u_0)}(\bm{E}(u_0))=\frac{d}{du}f_{\lambda(u_0)}(\bm{E}(u_0))=0.
\end{align*}
We now show that under the given hypotheses,
\begin{align*}
\frac{d^2}{du^2}f_{\lambda(u_0)}(\bm{E}(u_0))=0.
\end{align*}
For each $u\in U,$ there exist functions $q(\bm{e}(u))$ and $p(\bm{e}(u))$ such that
\begin{align*}
(-F_x(\bm{e}(u)),F_y(\bm{e}(u)))=q(\bm{e}(u))\mathbb{L}^++p(\bm{e}(u))\mathbb{L}^-.
\end{align*}
Using the expressions
\begin{align*}
\bm{\gamma}_t(\bm{e}(u))=\alpha(\bm{e}(u))\mathbb{L}^++\beta(\bm{e}(u))\mathbb{L}^-,\quad
\bm{\gamma}_\lambda(\bm{e}(u))=m(\bm{e}(u))\mathbb{L}^++n(\bm{e}(u))\mathbb{L}^-,
\end{align*}
we obtain that
\begin{align}\label{q}
\alpha(\bm{e}(u))q(\bm{e}(u))+\beta(\bm{e}(u))p(\bm{e}(u))=0
\end{align}
and
\begin{align*}
m(\bm{e}(u))q(\bm{e}(u))+n(\bm{e}(u))p(\bm{e}(u))=0.
\end{align*}
%Since the set
%of non-zero
%points of function $\ell\circ\bm{e}$ is dense in $U$, we have
%\begin{align}
%\cos\theta(\bm{e}(u))q(\bm{e}(u))+\sin\theta(\bm{e}(u))p(\bm{e}(u))=0.
%\end{align}
Since $\bm{e}$ is a pre-envelope of $\bm{\gamma}$, $n(\bm{e}(u))\cos\theta(\bm{e}(u))=m(\bm{e}(u))\sin\theta(\bm{e}(u))$ holds for all $u\in U.$ It follows $\alpha(\bm{e}(u))n(\bm{e}(u))-\beta(\bm{e}(u))m(\bm{e}(u))=0$ holds for all $u\in U.$ Differentiating with respect to $u$ gives
\begin{align}\label{d/du}
\frac{d}{du}\big(\alpha(\bm{e}(u))n(\bm{e}(u))-\beta(\bm{e}(u))m(\bm{e}(u))\big)=0.
\end{align}
Since $\ell\circ\bm{e}(u_0)\neq0$, we have $(\alpha(\bm{e}(u_0)),\beta(\bm{e}(u_0)))\neq(0,0)$, assume without loss of generality that $\alpha(\bm{e}(u_0))\neq0$.  By Equation (\ref{d/du}), we have
\begin{align*}
&\alpha(\bm{e}(u))A_1(u)
+\beta(\bm{e}(u))A_2(u)\\
=&-\frac{d\beta(\bm{e}(u))}{du}\alpha(\bm{e}(u))m(\bm{e}(u))+\frac{dn(\bm{e}(u))}{du}\alpha^2(\bm{e}(u))
+\frac{d\alpha(\bm{e}(u))}{du}\beta(\bm{e}(u))m(\bm{e}(u))\\
&-\frac{dm(\bm{e}(u))}{du}\alpha(\bm{e}(u))\beta(\bm{e}(u))\\
=&\alpha(\bm{e}(u))\frac{d}{du}\big(\alpha(\bm{e}(u))n(\bm{e}(u))-\beta(\bm{e}(u))m(\bm{e}(u))\big)\\
=&0,
\end{align*}
where
\begin{align*}
A_1(u)=-\frac{d\beta(\bm{e}(u))}{du}m(\bm{e}(u))+\frac{dn(\bm{e}(u))}{du}\alpha(\bm{e}(u)),\quad
A_2(u)=\frac{d\alpha(\bm{e}(u))}{du}m(\bm{e}(u))-\frac{dm(\bm{e}(u))}{du}\alpha(\bm{e}(u)).
\end{align*}
By Equation (\ref{q}), we have
\begin{align*}
\left(\begin{array}{cc}
        A_1(u)        & A_2(u)\\
        q(\bm{e}(u)) & p(\bm{e}(u))
      \end{array}
\right)
\left(\begin{array}{c}
        \alpha(\bm{e}(u)) \\
       \beta(\bm{e}(u))
      \end{array}
\right)=
\left(\begin{array}{c}
        0 \\
        0
      \end{array}
\right),
\end{align*}
Since the set
of non-zero
points of function $\ell\circ\bm{e}$ is dense in $U$, it means that the set $\{u\in U\mid\big(\alpha(\bm{e}(u)),\beta(\bm{e}(u))\big)\neq(0,0)\}$ is dense in $U$. Then
\begin{align*}
m(\bm{e}(u))A_3(u)
-\alpha(\bm{e}(u))A_4(u)
=-p(\bm{e}(u))A_1(u)+q(\bm{e}(u))A_2(u)=0,
\end{align*}
where
\begin{align*}
A_3(u)=\frac{d\alpha(\bm{e}(u))}{du}q(\bm{e}(u))+\frac{d\beta(\bm{e}(u))}{du}p(\bm{e}(u)),\quad
A_4(u)=\frac{dn(\bm{e}(u))}{du}p(\bm{e}(u))+\frac{dm(\bm{e}(u))}{du}q(\bm{e}(u)).
\end{align*}
Because $u=u_0$ is a singular point of $\bm{E}$, from the proof of Theorem \ref{T3}, we obtain that
\begin{align}\label{singular p}
\alpha(\bm{e}(u))t'(u)+m(\bm{e}(u))\lambda'(u)=0, \quad
\beta(\bm{e}(u))t'(u)+n(\bm{e}(u))\lambda'(u)=0
\end{align}
at $u=u_0$. Then, we have
\begin{align*}
\left(\begin{array}{cc}
        A_4(u) & -A_3(u) \\
        t'(u) & \lambda'(u)
      \end{array}
\right)
\left(\begin{array}{c}
        \alpha(\bm{e}(u)) \\
       m(\bm{e}(u))
      \end{array}
\right)=
\left(\begin{array}{c}
        0 \\
        0
      \end{array}
\right)
\end{align*}
at $u=u_0$. Because $\alpha(\bm{e}(u_0))\neq0$, the following equation holds at $u=u_0$,
\begin{equation}\label{t}
\begin{split}
\lambda'(u)\bigg(\frac{dn(\bm{e}(u))}{du}p(\bm{e}(u))+\frac{dm(\bm{e}(u))}{du}q(\bm{e}(u))\bigg)
+t'(u)\bigg(\frac{d\alpha(\bm{e}(u))}{du}q(\bm{e}(u))+\frac{d\beta(\bm{e}(u))}{du}p(\bm{e}(u))\bigg)=0.
\end{split}
\end{equation}
On the other hand, a straightforward computation shows that
\begin{align*}
\frac{d^2}{du^2}f_{\lambda(u_0)}(\bm{E}(u_0))=0
\end{align*}
holds if and only if
\begin{align}\label{la}
\frac{d}{du}\big(\alpha(\bm{e}(u))q(\bm{e}(u))+\beta(\bm{e}(u))p(\bm{e}(u))\big)t'(u)
+\frac{d}{du}\big(m(\bm{e}(u))q(\bm{e}(u))+n(\bm{e}(u))p(\bm{e}(u))\big)\lambda'(u)=0
\end{align}
holds at $u=u_0$. By Equation (\ref{singular p}), condition (\ref{la}) at $u=u_0$ is equivalent to condition (\ref{t}) at $u=u_0$. Thus, all required conditions are satisfied, and we conclude that
\begin{align*}
f_{\lambda(u_0)}(\bm{E}(u_0))=\frac{d}{du}f_{\lambda(u_0)}(\bm{E}(u_0))=\frac{d^2}{du^2}f_{\lambda(u_0)}(\bm{E}(u_0))=0,
\end{align*}
i.e., $\bm{E}$ and $\Gamma_{\lambda_0}$ have at least $3$-point contact at $u=u_0$.
\end{proof}

%%%%%%%%%%%%%%%%%%%%%%%%%%%%%%%%%%%%%%%%%%%
\section*{Acknowledgement}
%\begin{acknowledgements}
%\end{acknowledgements}
{
~~~This work was completed during the visit of the second author to Muroran Institute of Technology.
We would like to thank Muroran Institute of Technology for their kind hospitality. The first author was partially supported by
JSPS KAKENHI (Grant No. 24K06728). The second author was supported by Liaoning Provincial Science and Technology Department Project (Grant No. 2025-MSLH-096).

}
%
%
%%%%%%%%%%%%%%%%%%%%%%%%%%%%%%%%%%%%%%%%%%%%%%%%%%%%


\begin{thebibliography}{99}

\bibitem{BG1}
Bruce J.W., Giblin P.J., \emph{What is an envelope?}  Math. Gaz., {\bf65} (1981),  186-192.


\bibitem{brucegiblin}Bruce J.W., Giblin P.J.,
\emph{Curves and Singularities}. A Geometrical Introduction
to Singularity Theory, 2nd edn.
Cambridge University Press, Cambridge, 1992.

\bibitem{CT}
Chen L., Takahashi M.,
\emph{Lightcone framed curves in the Lorentz-Minkowski 3-space.}
Turk. J. Math., {\bf 48} (2024), 307--326.



\bibitem{EFK}
Ei S., Fujii K., Kunihiro T.,
\emph{Renormalization-group method for reduction of
evolution equations; invariant manifolds and envelopes.}
Ann. Phys., {\bf 280} (2000), 236--298.

\bibitem{fukunagatakahashi}
Fukunaga T., Takahashi M.,
\emph{Existence and uniqueness for Legendre curves}, J.~Geom.,
{\bf 104} (2013), 297--307.
%https://doi.org/10.1007/s00022-013-0162-6

\bibitem{GAS}
Gray A., Abbena E., Salamon S.,
\emph{Modern Differential Geometry of Curves
and Surfaces with Mathematica}, Studies in Advanced Mathematics, 3rd edn.
Chapman and Hall/CRC, Boca Raton (2006).




\bibitem{HST}Honda A., Saji K., Teramoto K.,
\emph{Mixed type surfaces with bounded Gaussian
curvature in three-dimensional Lorentzian
manifolds},
Adv. Math., {\bf 365} (2020), 107036.





\bibitem{izumiya1}
Izumiya S., \emph{Singular solutions of first-order differential equations},
 Bull. Lond. Math. Soc., {\bf 26} (1994), 69--74.



\bibitem{izumiya}
Izumiya S., Romero Fuster M.C., Takahashi M.,
\emph{Evolutes of curves in the Lorentz-Minkowski plane},
Adv. Stud. Pure Math.,
{\bf 78} (2018), 313--330.

%\bibitem{LPTY}
%Li Y., Pei D., Takahashi M., Yu H., \emph{Envelopes of Legendre curves in the unit spherical bundle over the unit sphere}, Quart. J. Math., {\bf 69} (2018), 631--653.



%\bibitem{izumiya}S.~Izumiya,
%\emph{Singular solutions of first order differential equations},
%Bull. London Math. Soc.,
%{\bf 26} (1994), 69--74.
%https://doi.org/10.1112/blms/26.1.69
\bibitem{LP}
Liu T., Pei D., \emph{Mixed-type curves and the lightcone
frame in Minkowski 3-space}, Int. J Geom. Methods M.,
(2020), 2050088.


\bibitem{nishimura}Nishimura T.,
\emph{Hyperplane families creating envelopes},
Nonlinearity,  {\bf 35} (2022),  2588.
%https://doi.org/10.1088/1361-6544/ac61a0

\bibitem{Takahashi1}
Takahashi M.,
\emph{On completely integrable first order ordinary differential equations.},
In: Proceedings of the Australian-Japanese Workshop on Real and Complex Singularities, (2007) pp.388-418.

\bibitem{Takahashi}
Takahashi M.,~
\emph{Envelopes of Legendre curves in the unit tangent bundle
over the Euclidean plane}, Results Math., {\bf 71} (2017),  1473--1489.

%\bibitem{TY}
%Takahashi M., Yu H.,~
%\emph{Envelopes of families of framed surfaces and singular solutions of first-order partial differential equations}, Proc. R. Soc. Edinb. A, {\bf 151} (2021),  1515--1542.



\end{thebibliography}
\end{document}